\documentclass[11pt, reqno]{amsart}
\usepackage{amssymb}
\usepackage{amsmath}
\allowdisplaybreaks[1]
\usepackage{graphicx,color}
\usepackage{wrapfig,framed}

\usepackage{microtype}
\usepackage[height=22.7cm, width=17cm, hmarginratio={1:1}]{geometry}
\usepackage{hyperref}

\theoremstyle{plain}
\newtheorem{thm}{Theorem}
\newtheorem{prop}{Proposition}
\newtheorem{lemma}{Lemma}
\newtheorem{cor}{Corollary}
\newtheorem{example}{Example}
\theoremstyle{definition}
\newtheorem{defi}{Definition}
\newtheorem*{remark}{Remark}

\newcommand{\beq}{\begin{equation}}
\newcommand{\eeq}{\end{equation}}
\newcommand{\ba}{\begin{align}}
\newcommand{\ea}{\end{align}}
\newcommand{\nn}{\nonumber}
\newcommand{\llll}{\langle\hskip -0.04truecm\langle}
\newcommand{\rrrr}{\rangle\hskip -0.04truecm\rangle}
\newcommand{\bll}{\bigl\langle\hskip -0.05truecm \bigl\langle}
\newcommand{\brr}{\bigr\rangle\hskip -0.05truecm \bigr\rangle}

\newcommand{\bllm}{\bigl[\hskip -0.04truecm \bigl[}
\newcommand{\brrm}{\bigr]\hskip -0.04truecm \bigr]}

\newcommand{\pal}{\partial}

\newcommand{\F}{\mathcal{F}}
\newcommand{\al}{\alpha}

\newcommand{\bv}{{\bf v}}
\newcommand{\bw}{{\bf w}}

\newcommand{\QQ}{{\mathbb Q}}
\newcommand{\CC}{{\mathbb C}}

\newcommand{\ZZ}{{\mathbb Z}}
\newcommand{\NN}{{\mathbb N}}

\newcommand{\R}{{\mathcal R}}
\newcommand{\A}{{\mathcal A}}
\newcommand{\C}{{\mathcal C}}

\newcommand{\pf}{\noindent{\it Proof \ }}
\newcommand{\tr}{{\rm tr}}

\newcommand{\bt}{{\bf t}}
\newcommand{\bT}{{\bf T}}
\newcommand{\e}{\epsilon}
\newcommand{\p}{\partial}

\newcommand{\epf}{$\quad$\hfill
\raisebox{0.11truecm}{\fbox{}}\par\vskip0.4truecm}

\def\={\; = \;}
\def\+{\; + \;}
\def\:={\; := \; }

\makeatletter
\newenvironment{subtheorem}[1]{
  \def\subtheoremcounter{#1}
  \refstepcounter{#1}
  \protected@edef\theparentnumber{\csname the#1\endcsname}
  \setcounter{parentnumber}{\value{#1}}
  \setcounter{#1}{0}
  \expandafter\def\csname the#1\endcsname{\theparentnumber.\Alph{#1}}
  \ignorespaces
}{
  \setcounter{\subtheoremcounter}{\value{parentnumber}}
  \ignorespacesafterend
}
\makeatother
\newcounter{parentnumber}

\begin{document}

\title{Gromov--Witten invariants of the Riemann sphere}
\author{Boris Dubrovin, Di Yang, Don Zagier}
\date{}
\dedicatory{
} 
\maketitle
\begin{abstract}
A conjectural formula for the $k$-point generating function of 
Gromov--Witten invariants of the Riemann sphere for all genera 
and all degrees was proposed in \cite{DY2}. In this paper, we give 
a proof of this formula together with an explicit analytic (as opposed to formal) expression for
the corresponding matrix resolvent. We also give a  formula for the $k$-point function
as a sum of $(k-1)!$ products of hypergeometric functions of one variable.
We show that the $k$-point generating function coincides with the $\e\rightarrow 0$ 
asymptotics of the analytic $k$-point function, and also compute three more asymptotics of the 
analytic function for $\e\rightarrow \infty$, $q\rightarrow 0$, $q\rightarrow\infty$, thus  
defining new invariants for the Riemann sphere.
\end{abstract} 

\medskip

\setcounter{tocdepth}{1}
\tableofcontents

\section{Statements of the main results}

\subsection{Gromov--Witten invariants of $\mathbb{P}^1$}
Let $\overline{\mathcal{M}}_{g,k}(\mathbb{P}^1,\beta)$ be the moduli space of stable maps 
from algebraic curves of genus $g$ with $k$ distinct marked points
 to $\mathbb{P}^1$, 
 of degree $\beta\in H_2(\mathbb{P}^1;\mathbb{Z})$
$$
\overline{\mathcal{M}}_{g,k}(\mathbb{P}^1,\beta) \= 
\bigl\{\, f: \, \left(\Sigma_{g}, p_1,\dots,p_k\right)\rightarrow \mathbb{P}^1 \, \big| \,~ f_*\left([\Sigma_g] \, \right) = \beta \, \bigr\} \, / \, \sim \, .
$$ 
Here, $(\Sigma_{g},p_1,\dots,p_k)$ denotes an algebraic curve of genus $g$ with at most double-point singularities and with 
 the distinct marked points $p_1,\dots,p_k$, and the equivalence relation $\sim$ is defined by isomorphisms of~$\Sigma_g\rightarrow \mathbb{P}^1$ 
identical on~$\mathbb{P}^1$ and on the markings.  
Let $\mathcal{L}_i$ be the $i^{\rm th}$ tautological line bundle on $\overline{\mathcal{M}}_{g,k}(\mathbb{P}^1,\beta)$, and 
 $\psi_i:= c_1(\mathcal{L}_i),$ $i=1,\dots,k$.  Denote by ${\rm ev}_i :   \overline{\mathcal{M}}_{g,k}(\mathbb{P}^1,\beta)\rightarrow \mathbb{P}^1$ the $i^{\rm th}$ evaluation map.

The genus $g$, degree $\beta$ Gromov--Witten (GW) invariants
of $\mathbb{P}^1$ are integrals of the form
\beq\label{corr-gd}
\int_{\left[\overline{\mathcal{M}}_{g,k}(\mathbb{P}^1,\beta)\right]^{{\rm virt}}} 
{\rm ev}_1^*(\phi_{\alpha_1})  \cdots  {\rm ev}_k^*(\phi_{\alpha_k}) \, \psi_1^{i_1} \cdots \psi_k^{i_k} \; =: \;  
\langle \tau_{i_1}(\phi_{\alpha_1}) \cdots \tau_{i_k}(\phi_{\alpha_k})  \rangle_{g,d}  \,.
\eeq
Here, $\alpha_1,\dots,\alpha_k\in\{1,2\},\, i_1,\dots,i_k\geq 0$,  $\phi_1=1$, $\phi_2=\omega\in H^2(\mathbb{P}^1; \mathbb{C})$ normalized by
$\int_{\mathbb{P}^1} \omega \= 1$,
and $\bigl[\overline{\mathcal{M}}_{g,k}(\mathbb{P}^1,\beta)\bigr]^{{\rm virt}}$ denotes the virtual fundamental class \cite{KM,B,BF,LT}.  
In the right-hand-side of equation \eqref{corr-gd}, the ``degree" $\beta\in H_2\bigl( \mathbb P^1; \mathbb Z\bigr)$ has been replaced by
an integer~$d$ through $d:=\int_\beta \omega$.
The GW invariant $\langle \tau_{i_1}(\phi_{\alpha_1}) \cdots \tau_{i_k}(\phi_{\alpha_k})  \rangle_{g,d}$ vanishes unless 
the degree--dimension matching holds:
$2g - 2 + 2d + 2 k = \sum_{\ell=1}^k i_\ell + \sum_{\ell=1}^k \alpha_\ell$.

For $k\geq 1$ and  $i_1,\dots,i_k\geq0$, $\alpha_1,\dots,\alpha_k \in\{1,2\}$, denote
\begin{align}
\langle \tau_{i_1}(\phi_{\alpha_1}) \cdots \tau_{i_k}(\phi_{\alpha_k})  \rangle \= \langle \tau_{i_1}(\phi_{\alpha_1}) \cdots \tau_{i_k}(\phi_{\alpha_k})  \rangle (\epsilon,q)
\:= \sum_{g=0}^\infty \sum_{d=0}^\infty  \epsilon^{2g-2} q^d \, \langle \tau_{i_1}(\phi_{\alpha_1}) \cdots \tau_{i_k}(\phi_{\alpha_k})  \rangle_{g,d} \,. \nn
\end{align}
We will call $\langle \tau_{i_1}(\phi_{\alpha_1}) \cdots \tau_{i_k}(\phi_{\alpha_k}) \rangle$ the~{\it $k$-point $\mathbb{P}^1$ correlator}, and 
$\langle \tau_{i_1}(\phi_{\alpha_1}) \cdots \tau_{i_k}(\phi_{\alpha_k})  \rangle_{g,d}$ 
the {\it $k$-point $\mathbb{P}^1$ correlator of genus~$g$ and degree~$d$}. 
Due to the degree--dimension matching, 
$\e^2\, \langle \tau_{k_1}(\phi_{\alpha_1}) \cdots \tau_{i_k}(\phi_{\alpha_k})  \rangle (\epsilon, q)$ is a homogeneous polynomial of $\epsilon^2, \,q.$  More precisely, 
$$
\langle \tau_{i_1}(\phi_{\alpha_1}) \cdots \tau_{i_k}(\phi_{\alpha_k})  \rangle (\epsilon,q)
\= \sum_{g,d\geq 0 \atop 2g+2d -2 = \sum_{\ell=1}^k (i_\ell + \alpha_\ell-2)}  \epsilon^{2g-2} q^d \langle \tau_{i_1}(\phi_{\alpha_1}) \cdots \tau_{i_k}(\phi_{\alpha_k})  \rangle_{g,d} \,.
$$
Note that this expression vanishes if $\sum_{\ell=1}^k (i_\ell + \alpha_\ell)$ is odd.

\begin{defi}  
The {\it free energy} $\F$ is defined as the following generating series of $\mathbb{P}^1$ correlators 
\beq\label{freeenergy}
\F \= \F(\bT;\e,q) \:= \sum_{k\geq 0} \frac1{k!} 
\sum_{1\leq \al_1,\dots,\al_k\leq 2 \atop i_1,\dots,i_k\geq 0}  T_{i_1}^{\alpha_1} \dots T_{i_k}^{\alpha_k} \,
\langle \tau_{i_1} (\alpha_1) \cdots \tau_{i_k} (\alpha_k)  \rangle(\e,q) 
\eeq
where $\bT=(T^\alpha_j)_{\alpha=1,2,\,j\geq 0}$. 
The {\it partial $k$-point correlation functions} are the power series 
$$\bll \tau_{i_1}(\phi_{\alpha_1}) \cdots \tau_{i_k}(\phi_{\alpha_k}) \brr (x; \e, q) \:= 
\frac{\p^k \F (\bT; \e, q)}{\p T^{\alpha_1}_{i_1} \dots \p T^{\alpha_k}_{i_k}}\bigg|_{\,T^\alpha_i \,=\, \delta_{\alpha 1} \, \delta_{i0} \, x}\,.$$
\end{defi}

Clearly,
$\bll \tau_{i_1}(\phi_{\alpha_1}) \cdots \tau_{i_k}(\phi_{\alpha_k}) \brr (0; \e, q) = \langle \tau_{i_1}(\phi_{\alpha_1}) \cdots \tau_{i_k}(\phi_{\alpha_k}) \rangle(\e,q)$. 
In this paper, we consider in particular the partial correlation-functions of the form
$\llll  \tau_{i_1}(\omega) \cdots \tau_{i_k}(\omega) \rrrr(x;\e, q)$, and consider 
 the following generating series \cite{DY2}, called the {\it $k$-point function}: 
\begin{align}
\label{mainsum}
& F_k(\lambda_1,\dots,\lambda_k; x; \e,q) \:= \e^k \! \sum_{i_1,\dots,i_k\geq 0}  
\frac{(i_1+1)! \cdots (i_k+1)!}{\lambda_1^{i_1+2} \dots \lambda_k^{i_k+2}} 
\bll  \tau_{i_1}(\omega) \cdots \tau_{i_k}(\omega)  \brr(x;\e,q)\;\quad (k\geq 1) \,.
\end{align}
Here $\lambda_1$, $\lambda_2$, $\dots$ are indeterminates. 
The dependence on $q$ in $F_k(\lambda_1,\dots,\lambda_k; x ; \e;q)$ can be recovered from $F_k(\lambda_1,\dots,\lambda_k; x; \e;1)$ by rescaling:
\beq \label{scaling2} 
F_k(\lambda_1,\dots,\lambda_k; x; \e, q) \equiv q^{-k/2} F_k \bigl( q^{-1/2} \lambda_1 ,\dots, q^{-1/2} \lambda_k; \, q^{-1} x; \, q^{-1/2} \e; 1 \bigr)  \;\quad (k\geq 1)\,.
\eeq
In particular, $F_k(\lambda_1,\dots,\lambda_k; 0; \e, q) \equiv q^{-k/2} F_k \bigl( q^{-1/2} \lambda_1 ,\dots, q^{-1/2} \lambda_k; 0; \, q^{-1/2} \e; 1 \bigr)$.

\subsection{The $k$-point function in terms of matrix resolvents } 
The matrix resolvent (MR) approach of computing logarithmic derivatives of tau-functions 
of continuous integrable systems
was introduced in \cite{BDY1,BDY2,BDY3}. 
It was further extended in \cite{DY1} to discrete integrable systems. 
The {\it Toda conjecture} (now a theorem) \cite{Du1, EY, Ge,OP,CDZ} says that $e^{\F}$ is the tau-function of a particular solution (which will be called the GW solution) to the 
Toda Lattice Hierarchy. So we can apply the MR approach \cite{DY1} to the computation of the $\mathbb{P}^1$ correlators. 
\begin{defi} [\cite{DY1, DY2}] \label{defiRn}
Let $U_n(\lambda;\epsilon)= \begin{pmatrix} \epsilon\,n + \frac\e2 \,-\,\lambda & 1\\ -1 & 0 \end{pmatrix}$.
Define the {\it matrix resolvent} $R_n(\lambda;\epsilon)$ for the GW solution of the Toda Lattice Hierarchy as the unique formal solution to the following problem
\begin{align}
& R_{n+1}(\lambda;\epsilon) \, U_n(\lambda;\epsilon) \, - \,  U_n(\lambda;\epsilon) \, R_n(\lambda;\epsilon) \= 0\,, \label{r1} \\
& {\rm \tr} \, R_n(\lambda;\epsilon) \= 1,\qquad {\rm \det} \, R_n(\lambda;\epsilon) \= 0\,,  \label{r2} \\
& R_n(\lambda;\epsilon) \= \begin{pmatrix} 1 & 0  \\  0 & 0 \\ \end{pmatrix}
\+ {\rm O}\bigl(\lambda^{-1}\bigr)  \,, \quad \lambda\rightarrow \infty\,. \label{r3}
\end{align}
\end{defi}
\noindent This solution $R_n(\lambda;\e)$ belongs to ${\rm Mat}_2\bigl(\mathbb{Z}[n,\e]\bllm\lambda^{-1}\brrm\bigr)$.
Define $\R(\lambda;x;\e):=R_{x/\epsilon}(\lambda;\epsilon)$. 

\begin{thm} \label{prop1} 
The formal series \eqref{mainsum} with $k\geq 2$ have the expressions
\begin{align}
& F_2(\lambda_1,\lambda_2;x;\e,1) \= \frac{\tr \, \left[  \R(\lambda_1;x;\e) \,  \R(\lambda_2;x;\e)\right]-1}{(\lambda_1-\lambda_2)^2}\, , \label{twopoint} \\
& F_k(\lambda_1, \dots, \lambda_k;x;\e,1) \= -\sum_{\sigma\in S_k/C_k} \frac{\tr \,\left[{\mathcal R}(\lambda_{\sigma(1)};x;\e)\dots {\mathcal R}(\lambda_{\sigma(k)};x;\e)\right]}
{\prod_{i=1}^k (\lambda_{\sigma(i)}-\lambda_{\sigma(i+1)})} \,, \quad k\geq 3 \,. \label{npoint}
\end{align}
Here $S_k$ and $C_k$ are the symmetric group and standard 
cyclic subgroup,  with $\sigma(k+1)=\sigma(1)$ for $\sigma\in S_k$.   
\end{thm}
\noindent The proof, based on the Toda conjecture,  
uses a simple observation \cite{DY2} and the MR approach \cite{DY1}. 
The idea of the proof has been explained in \cite{DY2}; we provide the details in Section \ref{section3} of the current paper. 

The following property, proved in Section \ref{sec3}, is related to the concept of {\it bispectrality} (see e.g.~\cite{DG}).
\begin{thm} \label{pbispec}
The matrix-valued formal series $\R(\lambda;x;\e)$ depends only on $\e$ and $\lambda-x$. 
\end{thm}
 In other words, $\R(\lambda;x;\e)$ has the form 
\beq
\label{RDon}
\R(\lambda;x;\e) \= M\Bigl(\frac{\lambda-x}\e; \frac1\e\Bigr)
\eeq
for some $M(z;s)$, which is a formal power series in $z^{-1}$. On the other hand, from its definition, $\R(\lambda;x;\e)$ satisfies 
$$
\R(\lambda; x+\e; \epsilon) \, \begin{pmatrix} x+\frac\e 2-\lambda & 1\\ -1 & 0\end{pmatrix} 
\, - \, \begin{pmatrix} x+\frac\e 2-\lambda & 1\\ -1 & 0\end{pmatrix} \, \R(\lambda; x; \epsilon) \= 0
$$
which in terms of $M(z;s)$ becomes
\beq\label{m1xp} 
M(z-1;  s) \begin{pmatrix} z-\frac12 & -s\\ s & 0\end{pmatrix} \=  \begin{pmatrix} z-\frac12 & -s\\ s & 0\end{pmatrix} \, M(z;s) \,.
\eeq
Similarly, from equations \eqref{r2} and \eqref{r3} we deduce that $M(z;s)$ also satisfies
\beq\label{m1xp1} 
\tr \, M(z;s) \= 1 \,, \qquad M(\infty;s) \= \begin{pmatrix} 1 & 0  \\  0 & 0 \\ \end{pmatrix}\,.
\eeq
We call \eqref{m1xp} the {\it topological difference equation}, which is an analogue of the topological ODE \cite{BDY2,BDY3}.
\begin{prop}\label{ptopod}  
There exists a unique element $M^*$ in $M_2\bigl(\CC(s) \bllm z^{-1}\brrm\bigr)$ satisfying equations \eqref{m1xp}--\eqref{m1xp1}.  Moreover, 
$M^*$ belongs to $M_2\bigl(\QQ[s] \bllm z^{-1}\brrm\bigr)$, and it satisfies $\det M^*=0$.
\end{prop}  
\noindent 
See Section \ref{sec3} for the proof. 
Proposition \ref{ptopod} will be used to prove Theorem~\ref{pbispec}, with $M=M^*$ (see equation~\eqref{RDon} above).

The following theorem, which was conjectured in~\cite{DY2}, gives an explicit formula for the
matrix $M(z;s)$.  We will prove it in Section~\ref{section3}. (A different proof was given recently by O.~Marchal \cite{M}.)

\begin{thm} \label{conj-corr} 
The matrix-valued power series $M=M(z;s)$ has the following explicit expression
\beq\label{ram}
M \=  \begin{pmatrix} 
1+\alpha & Q - P\\
Q + P & -\alpha
\end{pmatrix} \, , 
\eeq
where $\alpha=\alpha(z;s),\, P=P(z;s), \, Q=Q(z;s) \, \in \QQ[s]\bllm z^{-1}\brrm$ are given by 
\begin{align}
& \alpha(z;s) \= 2 \sum_{j=0}^\infty \frac{1}{z^{2j+2}} \sum_{i=0}^j   s^{2i+2}  \, \frac{1}{ i! (i+1)!} \, 
\sum_{\ell=0}^i  (-1)^\ell \bigl(i-\ell+\tfrac12\bigr)^{2j+1} \binom{2i+1}{\ell} \, , \label{defofalpha} \\
& P(z;s) \= \sum_{j=0}^\infty \frac{1}{z^{2j+1}} \sum_{i=0}^j   s^{2i+1}  \, \frac{1}{i!^2} \, 
\sum_{\ell=0}^i  (-1)^\ell \bigl(i-\ell+\tfrac12\bigr)^{2j} \biggl[\binom{2i}{ \ell } - \binom{2i}{\ell-1}\biggr] \, , \\
& Q(z;s) \=  - \frac12 \sum_{j=0}^\infty \frac{1}{z^{2j+2}} \sum_{i=0}^j   s^{2i+1}  \, \frac{2i+1}{i!^2} \, 
\sum_{\ell=0}^i  (-1)^\ell \bigl(i-\ell+\tfrac12\bigr)^{2j} \biggl[\binom{ 2i }{ \ell } - \binom{ 2i}{ \ell-1} \biggr] \, .
\end{align}
\end{thm}

\subsection{Explicit formulas in terms of hypergeometric functions and Bessel functions}  \label{sec1.3}
Define a meromorphic matrix-valued function $B=B(z;s)$ by 
\beq\label{defBzs}
B(z;s) \= \dfrac12 
\begin{pmatrix} 
1+ G\bigl(z, s\bigr) &  \frac{4s}{1-2z} \widetilde G\bigl(z -1 , s\bigr)\\  
\\
\frac{4s}{1+2z } \widetilde G\bigl(z , s\bigr) & 1-  G\bigl(z , s\bigr)\\ \end{pmatrix}\,, \quad z\in \CC-\mathbb{Z}_{\rm odd}\,,~s\in \CC
\eeq
where $G(z;s)$ and $\widetilde G(z;s)$ are the (generalized) hypergeometric functions
\begin{align} 
& G(z;s) \=  {}_1F_2\Bigl(\frac12; \frac12-z, \frac12+z; - 4 s^2\Bigr) 
\= \sum_{m=0}^\infty \binom{2m}{m}  \frac{s^{2m}}{(z-m+\frac12)_{2m}} \,, \label{Hyper1} \\
& \widetilde G(z;s) \=    {}_1F_2  \Bigl(\frac12; \frac12-z, \frac32+z; - 4 s^2\Bigr) 
\= \sum_{m=0}^\infty \binom{2m}{m} \frac{z+\frac12}{(z-m+\frac12)_{2m+1}}  s^{2m} \,. \label{Hyper2}
\end{align}
Here, $(\,)_k$ is the increasing Pochhammer symbol, i.e. $(x)_k := \Gamma(x+k)/\Gamma(x) =x (x+1) \cdots (x+k-1)$. 
Note that the series in~\eqref{Hyper1}, \eqref{Hyper2} converge 
absolutely and locally uniformly away from $z\in \ZZ+\frac12$ so that the product 
$\cos (\pi z ) \, B(z;s)$ extends to a  (matrix-valued) holomorphic function on all of $\CC^2$.

\begin{thm}
\label{thm05} 
For fixed $s\in\CC$, the asymptotic expansion of $B(z;s)$ in all orders as $z\rightarrow \infty$ at a bounded distance from $\ZZ+\frac12$ coincides with the
formal power series $M(z;s)$, i.e. $B(z;s) \, \sim \,  M(z;s)$.
\end{thm}
\noindent The proof will be given in Section \ref{sec3}.   Here and throughout this paper, we use~$\sim$ 
for a full asymptotic expansion: e.g. ``$f(\e)\sim g(\e)$ as $\e\to0$" means that the asymptotic expansions of $f$ and~$g$ 
 agree as power series  in~$\e$ to all orders. 
 
We now define {\it analytic $k$-point functions}  $H_k(z_1, \dots, z_k;s)$ ($k\geq2$; the case $k=1$ will be treated later) by
\begin{align}
& H_2(z_1,z_2;s) \:= \frac{\tr \, \bigl[  B(z_1;s) \,  B(z_2;s) \bigr]-1}{(z_1-z_2)^2} \= - \frac12 \,  \tr \,  \biggl[  \frac{ B(z_1;s) - B(z_2;s)}{z_1-z_2}\biggr]^2 , \label{twopointholo} \\
& H_k(z_1, \dots, z_k;s) \:= -\sum_{\sigma\in S_k/C_k} \frac{\tr \,\left[B(z_{\sigma(1)};s)\cdots B(z_{\sigma(k)};s)\right]}
{\prod_{i=1}^k (z_{\sigma_i}-z_{\sigma_{i+1}})} \;.\label{npointholo}
\end{align}
Then equation \eqref{npoint00} and Theorem \ref{thm05} imply that the full asymptotic expansion of $H_k\bigl(\frac{\lambda_1}\e, \dots, \frac{\lambda_k}\e; \frac1\e\bigr)$ 
coincides with $F_k(\lambda_1,\dots,\lambda_k;\e)$ as $\lambda_i\to\infty$ at a bounded distance away from $\e \, \mathbb{Z} + \frac\e2$,
$i=1,\dots,k$.  
Notice that this
statement is not entirely trivial, since the presence of poles in~\eqref{twopointholo} and~\eqref{npointholo} when two $z_i$'s coincide means that a priori 
we must order the $|\lambda_i|$'s in order to obtain an asymptotic expansion.  However, the following proposition (which will be proved in Section \ref{sec3})
implies that the asymptotics are the same for all orderings of the~$|\lambda_i|$'s.

\begin{prop} \label{nopoleH}
The functions $H_k(z_1,\dots,z_k;s)$,~$k\geq 2$ are analytic along the diagonals $z_i=z_j\,, i\neq j$ (as above it is assumed that none of $z_1$, \dots, $z_k$ is a half-integer).
\end{prop}

The next point is very nice.  Normally, even if one knows a closed formula for a collection of $2\times2$ matrices, it is not easy to compute the trace
of their product.  Here, however, there is a nice simplification which lets us write the above traces as products.
The reason is that, since $\,\det B(z;s)=0$ by Proposition~\ref{ptopod} and Theorem~\ref{thm05}, the matrix $B(z;s)$ must 
factor as the product of a column vector and a row vector. This factorization, given explicitly in the following proposition,
will immediately lead us to a factorization formula for the traces. 
\begin{prop}\label{BBB} The matrix valued function $B$ has the following expressions
\beq\label{RDonM2}
B(z;s) \= u(z) \, u(-z)^T \=  \frac{\pi s} {\cos (\pi z) }  V(z) \, V(-z)^T \,,
\eeq 
where  
$$u(z)=u(z;s) = \begin{pmatrix} j_{z} (s^2) \\  \frac{s}{z+\frac12} j_{z+1} (s^2)
\end{pmatrix}\,, \qquad   V(z) = V(z;s) \= \begin{pmatrix} J_{z-\frac12}(2s) \\ J_{z+\frac12}(2s) \\ \end{pmatrix} \= \frac{s^{z-\frac12}}{\Gamma(z+\frac12)} \, u(z;s) \,.$$
Here $J_\nu(y)$ denotes the standard Bessel function \cite{Watson} and $j_a(X)$ a modified Bessel function:
\beq
j_a (X) \:= \sum_{n\geq 0}  \frac{(-X)^n}{n!^2 \, \binom{n+a-1/2}{n}} \,, \qquad 
J_\nu(y)\:=  \frac{(y/2)^\nu}{\Gamma(\nu+1)} \, j_{\nu+\frac12}(y^2/4) \,. 
\eeq
\end{prop} 

Now define two  analytic functions $D(a,b;s)$ and $D^*(a,b;s)$ by
\beq \begin{aligned}\label{defDabX}
D(a,b;s) & \= \frac{u(-a,s)^T \, u(b,s)}{a-b} \=  \frac{j_{-a}(X) \, j_{b}(X) \+ \frac{X}{(\frac12-a)(\frac12+b)} j_{1-a}(X) \, j_{1+b}(X) }{a-b}\, ,  \\
D^*(a,b;s) & \= \frac{V(-a,s)^T \, V(b,s)}{a-b} \= \frac{J_{-a-\frac12}(2s) \, J_{b-\frac12}(2s) \+ J_{\frac12-a}(2s) \, J_{\frac12+b}(2s) }{a-b}  \\
& \= \frac{s^{b-a-1}}{\Gamma(\frac12-a) \, \Gamma(\frac12+b)} \, D(a,b;s) \,, 
\end{aligned} \eeq 
where $X=s^2$. Then from Proposition \ref{BBB} and the fact that $\, \tr (A_1B_1\cdots A_k B_k) = \tr (B_1 A_2 \cdots B_k A_1)$, we find that 
the trace in \eqref{npoint00} factorizes  as a product of the one-variable functions $D$ or~$D^*$, and we obtain:

\begin{thm} \label{newformulaana} 
The analytic functions $H_k,\,k\geq 2$ have the expressions
\begin{align}
H_k(z_1, \dots, z_k;s)
& \= - \sum_{\sigma\in S_k/C_k} 
\prod_{i=1}^k D\bigl(z_{\sigma(i)}, z_{\sigma(i+1)} ; s \bigr)  \,-\, \frac{\delta_{k,2}}{(z_1-z_2)^2} 
\end{align}
or alternatively
\beq
H_k(z_1, \dots, z_k;s) \=  - \frac{\pi^k \, s^k}{\prod_{i=1}^k \cos (\pi z_i) } \sum_{\sigma\in S_k/C_k} 
\prod_{i=1}^k D^*\bigl(z_{\sigma(i)}, z_{\sigma(i+1)} ; s \bigr)  \,-\, \frac{\delta_{k,2}}{(z_1-z_2)^2} \,.
\eeq
with $D(a,b;s)$ and $D^*(a,b;s)$ as in \eqref{defDabX}.
\end{thm}

\begin{example} The function $H_2$ has the expression:
\begin{align}
H_2(z_1,z_2;s)  \= - \, D\bigl(z_1,z_2;s\bigr) \, D\bigl(z_2,z_1;s\bigr) \,-\, \frac1{(z_1-z_2)^2} \,.
\end{align}
\end{example}

Next, we note that, although the original definition \eqref{defDabX} would give a complicated formula for $D(a,b;s)$ as a double infinite sum, 
in fact it simplifies to a single infinite sum (hypergeometric series):
\begin{prop} \label{dhyper}
The function $D(a,b;s)$ has the following explicit expression
\beq\label{Dhgm}
D(a,b;s) \= \sum_{n=0}^\infty \frac{(a-b-2n+1)_{n-1}}{n! \, (-a+\tfrac12)_n \, (b+\tfrac12)_n} s^{2n} \,, 
\eeq
where in the first term $(a-b+1)_{-1}:=1/(a-b)$. Equivalently, 
$$
D(a,b;s)\= \frac1{a-b} ~  {}_2 F_3\Bigl(\frac{b-a}2, \frac{b-a+1}2; \, \frac12 -a \,, \, \frac12 + b \,, b-a+1; \, -4 s^2\Bigr) \,.
$$
\end{prop}
\begin{proof} 
This follows from a product formula for Bessel functions given on p.~147 of \cite{Watson}.
\end{proof}

\subsection{One-point functions} \label{sec1.4}
In the above we looked at $k$-point functions with $k\geq 2$. We now consider the case~$k=1$. 
Define two meromorphic functions $H_1(z,s)$ and $H_1^*(z,s)$ as modified limiting functions of $D(a,b;s)$ and $D^*(a,b;s)$, namely  
\begin{align}
& H_1(z,s) \= - \lim_{b\to z} \bigl(D(z,b;s) \, -\, \tfrac1{z-b}\bigr) \,, \\
& H_1^*(z,s)\= -\frac{\pi \,s }{\cos(\pi z)} \lim_{b\to z} \Bigl(D^*(z,b;s) \, -\, \tfrac{\cos (\pi \, z)}{\pi\,s\,(z-b)}\Bigr) \,. 
\end{align}

From the definition it follows immediately that the functions $H_1$ and $H_1^*$ are related by
\beq \label{H1H1}
H_1^*(z; s) \= H_1(z; s) \+  \log s \, -\, \psi\Bigl(\frac12 + z \Bigr) \,,
\eeq
where $\psi$ denotes the digamma function.
Using equations \eqref{defDabX} and Proposition \ref{dhyper} along with l'Hospital's rule we get the following explicit expressions:
\begin{align}
& H_1(z;s) \= \sum_{n\geq 1} \frac{(2n-1)! \, s^{2n}}{n!^2 (z-n+\frac12)_{2n}} \,,\label{H1sum}\\
& H_1^*(z;s) \=   \frac{\pi \,s }{\cos(\pi z)} \, \biggl(J_{-\frac12-z}(2s) \, \frac{\p J_{-\frac12+z}(2s)}{\p z} \+ J_{\frac12-z}(2s) \, \frac{\p J_{\frac12+z}(2s)}{\p z} \biggr) \,. \label{H1starJ}
\end{align}
From \eqref{H1sum} one observes that $s \frac{\p H_1(z;s)}{\p s} \= G(z;s) \,- \, 1$. 

\begin{thm}\label{thm2}
The formal series \eqref{mainsum} with $k=1$ has the expression 
\beq\label{F1H1}
F_1(\lambda;x;\e,1) \;=\;  \frac1\e \Bigl(H_1^*\Bigl(\frac{\lambda-x}\e; \frac1\e\Bigr) \+ \log \lambda \,-\, \frac{x}\lambda \Bigr) \,,
\eeq
where the right-hand-side is understood as its asymptotic expansion as $\lambda\rightarrow \infty$.
\end{thm}
Alternatively, using \eqref{H1sum} and \eqref{H1H1}, we can write \eqref{F1H1} explicitly as
\begin{align}
& F_1(\lambda;x;\e,1) \= 
\frac1\e \sum_{j\geq 2} \frac{x^j}{j \, \lambda^j}
\+ \sum_{j\geq 2} \, \frac{\e^j}{j\,(\lambda-x)^j} \sum_{i=0}^\infty  \frac{\e^{-1-2i}}{i!^2} 
\sum_{\ell=0}^{2i} (-1)^{\ell} \, \binom{2i}{l}  \, B_j\bigl(i - \ell +\tfrac12 \bigr) \,,   \label{onepointaltern} \end{align}
which can also be written as a pure power series in $\lambda^{-1}$ as
\begin{align}
& F_1(\lambda;x;\e,1) \=  \sum_{j\geq 2} \frac{\e^j}{j\,\lambda^{j}} \sum_{i=0}^\infty  \frac{\e^{-1-2i}}{i!^2} \sum_{\ell=0}^{2i} (-1)^{\ell}  \, \binom{2i}{\ell} \, 
B_j \bigl(  \tfrac{x}\e + i -\ell +\tfrac12 \bigr) \,. \label{onepoint}
\end{align}
Here, $B_j(u)$ denotes the $j^{\rm th}$ Bernoulli polynomial  
(the unique polynomial solution to $\int_v^{v+1}\!B_j(u)du=v^j$).

Note that the internal sum in \eqref{onepointaltern} or \eqref{onepoint} is simply the $(2i)^{\rm th}$ backward-difference of the polynomial 
$B_j\bigl(i +\tfrac12 \bigr)$ or $B_j \bigl(  \tfrac{x}\e +i +\tfrac12 \bigr)$, 
respectively, and since the $n^{\rm th}$ difference of a polynomial of degree $<n$ vanishes, we can replace the sum~$\sum_{i=0}^\infty$ by 
$\sum_{i=0}^{[j/2]}$ in both equations. Also, since $\Delta_{-1}^{2i} \bigl(B_j\bigl(\frac{x}\e+i +\tfrac12 \bigr)\bigr) = \Delta_{-1}^{2i-1} \bigl(\Delta_{-1} \bigl(B_j\bigl(\frac{x}\e+i +\tfrac12 \bigr) \bigr)\bigr) =  j \, \Delta_{-1}^{2i-1}\bigl(\bigl(\frac{x}\e+i - \tfrac12\bigr)^{j-1}\bigr)$, we 
have the following more elementary expressions
\begin{align}
& F_1(\lambda;x;\e,1) \= 
 \sum_{g\geq 1} \, \frac{\e^{2g-1}}{(\lambda-x)^{2g}}  \frac{\bigl(1-2^{2g-1}\bigr) \, B_{2g}}{2^{2g} \, g }   \+ \frac1\e \sum_{j\geq 2} \frac{x^j}{j \, \lambda^j} \nn\\
&\qquad \qquad \qquad \quad \+ \sum_{j\geq 2} \, \frac{\e^j}{(\lambda-x)^j} \sum_{i=1}^{[j/2]}  \frac{\e^{-1-2i}}{i!^2} 
\sum_{\ell=0}^{2i-1} (-1)^{\ell} \, \binom{2i-1}{\ell}  \, \bigl(i - \ell - \tfrac12 \bigr)^{j-1} \,,   \nn\\
& F_1(\lambda;x;\e,1) \= \sum_{j\geq 2} \frac{\e^{j-1}}{j\,\lambda^{j}}  \, 
B_j \bigl(  \tfrac{x}\e +\tfrac12 \bigr)  \+  \sum_{j\geq 2} \frac{\e^j}{\lambda^{j}} \sum_{i=1}^{[j/2]}  \frac{\e^{-1-2i}}{i!^2} \sum_{\ell=0}^{2i-1} (-1)^{\ell}  \, \binom{2i-1}{\ell} \, 
 \bigl(  \tfrac{x}\e + i -\ell - \tfrac12 \bigr)^{j-1} \,, \nn
\end{align}
where the first sum in each case corresponds to the digamma term in \eqref{H1H1}, and $B_j:=B_j(0)$ is the $j^{\rm jh}$ Bernoulli number.

\subsection{Four asymptotics} 
We already know that $F_k(\lambda_1,\dots,\lambda_k; 0;\e,1)$ contains all GW invariants of $\mathbb{P}^1$ in the stationary sector for all genera and all degrees. 
(The dependence on $q$ can be recovered by rescalings). This suggests the possibility of studying the $\e\rightarrow \infty$ or 
$q\rightarrow \infty$ limit using the analytic $k$-point functions. 
Namely, we study the functions 
$H_k\bigl(\frac{\lambda_1}\e, \dots, \frac{\lambda_k}\e; \frac{q^{1/2}}\e\bigr)$, $k\geq 1$
and their four asymptotic behaviours: $\e\rightarrow 0$,  $\e \rightarrow \infty$, $q\rightarrow 0$, $q\rightarrow \infty$. 

For any fixed $k\geq 1$, introduce the grading operator 
${\rm gr} := \e \frac{\p}{\p \e} \+ 2q\frac{\p}{\p q} \+ \sum_{i=1}^k \lambda_i \frac{\p}{\p \lambda_i}$.
Obviously, 
$~ {\rm gr} \, H_k \bigl(\frac{\lambda_1}\e, \dots, \frac{\lambda_k}\e; \frac{q^{1/2}}\e\bigr)  = 0$.
We begin with the $\e\rightarrow 0$ limit. The condition $\lambda_i/2\sqrt{q}>1$ in the following theorem may look strange at first. It comes from the fact that 
the asymptotics of $J_\nu(\nu x)$ as $\nu\to \infty$ with $x>0$ fixed are different according as $x<1$, $x=1$, or $x>1$. (See~\cite{Watson}, p.~225.)
\begin{subtheorem}{thm} \label{thmfourlimits}
\begin{thm}\label{e0limit}
\textit{Fix $k\geq 1$. For $q,\e,\lambda_1,\dots,\lambda_k$ satisfying $0<\frac{2\sqrt{q}}{\lambda_i}<1$, $\frac{\lambda_i}\e>0$, $i=1,\dots k$, as~$\e\to 0$ (with fixed $\lambda_1,\dots,\lambda_k,q$), we have expansions of the form
\begin{align}
& H_k\Bigl(\frac{\lambda_1}\e, \dots, \frac{\lambda_k}\e; \frac{q^{1/2}}\e\Bigr)  \; \sim \; \sum_{g\geq 0} \e^{2g-2+2k} H_{k}^{[g]}(\lambda_1,\dots,\lambda_k;q)  
 \qquad (k\geq 2)\,, \label{asympHe}\\
& H_1^*\Bigl(\frac{\lambda}\e; \frac{q^{1/2}}\e\Bigr) \; \sim \;  \log q^{1/2} - \log \lambda \+   \sum_{g\geq 0} \e^{2g} H_1^{*,[g]}(\lambda;q) \,,
\end{align}
where $H_{k}^{[g]}$ for $k\geq 2$ are rational functions of $q,\lambda_1,\dots,\lambda_k$ satisfying ${\rm gr} \, H_{k}^{[g]} \= (2-2g-2k)  \,H_{k}^{[g]}$; 
while for $k=1$ we have $H_1^{*,[0]} = \log \frac{2\lambda}{\lambda+ (\lambda^2-4q)^{\frac12}}$ and $H_1^{*,[g]}$~$(g\geq1)$ has the form
\beq
\lambda^{2g} \, H_1^{*,[g]}(\lambda;q) \= \frac{P_g\bigl(\frac q {\lambda^2} \bigr)}{ \bigl(1-\frac{4q}{\lambda^2}\bigr)^{\frac{6g-1}2}} \,,\qquad 
P_g(0) = 2 \, \frac{(2g-1)! \,  (1-2^{2g-1}) \, B_{2g}}{4^g\, (2g)!}
\eeq 
with $P_g(x)$ being a polynomial with rational coefficients of degree $2g-1$.
Moreover, the sum 
$$\sum_{g\geq 0} \e^{2g-2+k} H_{k}^{[g]}(\lambda_1,\dots,\lambda_k;q) ~ (k\geq2) \quad 
\mbox{or}~\sum_{g\geq 0} \e^{2g-1} H_{1}^{*,[g]}(\lambda;q) ~(k=1)$$  
coincides with $F_k(\lambda_1,\dots,\lambda_k; 0; \e,q)$ as a formal power series in $\lambda_1^{-1},\dots,\lambda_k^{-1},q$.
}
\end{thm}

\noindent The proof will be given in Section \ref{sectionlimits}. 
Theorem \ref{e0limit} tells that the $\e\rightarrow 0$ limit gives GW invariants of $\mathbb{P}^1$.
We remark that even the simple consequence of \eqref{asympHe} that $H_k\bigl(\frac{\lambda_1}\e, \dots, \frac{\lambda_k}\e; \frac{q^{1/2}}\e\bigr) = {\rm O}(\e^{2k-2})$ 
as $\e\rightarrow 0$ already seems to be quite non-trivial.  
The first few $H_k^{[g]}$ (or $H_1^{*,[g]}$)  are given by
\begin{align}
& H_1^{*,[1]} \= - \frac{\lambda \, (\lambda^2-16q)}{24 \, (\lambda^2-4q)^{\frac52}} \,,\qquad H_1^{*,[2]} \= \frac{\lambda \, (7 \lambda^6 -94 q\lambda^4+8256 q^2 \lambda^2 + 18432 q^3) }{960 \, (\lambda^2-4q)^{\frac{11}2}} \,, \nn\\
& H_2^{[0]} \= \frac{\lambda_1 \lambda_2-\sqrt{\lambda_1^2-4 q} \sqrt{\lambda_2^2-4 q}-4 q}{2 (\lambda_1-\lambda_2)^2 \sqrt{\lambda_1^2-4 q} \sqrt{\lambda_2^2-4 q}} \,, \nn\\
& H_2^{[1]} \= \frac{q}{4 \bigl(\lambda_1^2-4 q\bigr)^{\frac72} \bigl(\lambda_2^2-4 q\bigr)^{\frac72}} \Bigl(\lambda_1^3 \lambda_2^3 \bigl(\lambda_1^2+\lambda_2^2\bigr) +4 q\lambda_1 \lambda_2 \bigl(4 \lambda_1^4+5 \lambda_1^3 \lambda_2-\lambda_1^2 \lambda_2^2+5 \lambda_1 \lambda_2^3+4 \lambda_2^4\bigr) \nn\\ 
& \qquad\qquad\qquad -16 q^2 \lambda_1 \lambda_2 \bigl(10 \lambda_1^2+17 \lambda_1 \lambda_2+10 \lambda_2^2\bigr) +64 q^3 \bigl(2 \lambda_1^2+11 \lambda_1 \lambda_2+2 \lambda_2^2\bigr) +768 q^4 \Bigr) \,, \nn\\
& H_3^{[0]} \= q \frac{\lambda_1 \lambda_2 \lambda_3  + 4q (\lambda_1 + \lambda_2 + \lambda_3)}{ (\lambda_1^2-4 q)^{\frac32} (\lambda_2^2-4 q)^{\frac32} (\lambda_3^2-4 q)^{\frac32}} \,.\nn
\end{align}

The next is to look at the $\e\rightarrow \infty$ limit of $H_k\bigl(\frac{\lambda_1}\e, \dots, \frac{\lambda_k}\e; \frac{q^{1/2}}\e\bigr)$. 

\begin{thm}\label{einftylimit} 
\textit{Fix $k\geq 1$.  $\forall\, \lambda_1,\dots,\lambda_k,q\in\CC$, the following asymptotic holds true: as $\e\rightarrow\infty$,
\beq\label{asympHeinf}
H_k\Bigl(\frac{\lambda_1}\e, \dots, \frac{\lambda_k}\e; \frac{q^{1/2}}\e\Bigr) \;  \sim \; \sum_{g\geq 0} \e^{-2g} H_{k,[g]}(\lambda_1,\dots,\lambda_k;q) \,,
\eeq
where $H_{k,[g]}\in \QQ[\lambda_1,\dots,\lambda_k,q]$, and $~{\rm gr} \, H_{k,[g]} = 2g  \,H_{k,[g]}$. }
\end{thm}

\noindent The sum in the RHS of \eqref{asympHeinf} converges if $|\e|> 2 \, {\rm max} \bigl\{ |\lambda_1|, \dots, |\lambda_k| \bigr\}$. 

Thirdly we look at the $q\rightarrow 0$ limit.
\begin{thm}\label{q0limit} 
\textit{Fix $k\geq 1$. For $\lambda_i \notin \e\,\ZZ + \frac{\epsilon}2$, $i=1,\dots,k$, as $q\rightarrow 0$,
\beq\label{asympHq}
H_k\Bigl(\frac{\lambda_1}\e, \dots, \frac{\lambda_k}\e; \frac{q^{1/2}}\e\Bigr)  \; \sim \; \sum_{d\geq 0} q^d \, H_{k,d}(\lambda_1,\dots,\lambda_k;\e)
\,, \eeq
where $H_{k,d}(\lambda_1,\dots,\lambda_k;\e)$ are rational functions of $\lambda_1,\dots,\lambda_k,\e$ with poles only at $\lambda_i=m\, \e/2$ with $|m|<2d$ odd,  
and ${\rm gr} \, H_{k,d} = -2d  \,H_{k,d}$. The 
$q\to 0$ asymptotic of $H_1$ has the explicit expression
\beq\label{H1q0asym}
H_1\Bigl(\frac{\lambda}\e; \frac{q^{1/2}}\e\Bigr)  \; \sim \; \sum_{d=1}^\infty q^d \,  \frac{(2d-1)!}{d!^2 \, \prod_{j=1}^d \bigl(\lambda^2-\tfrac{(2j-1)^2}{4}\e^2\bigr)} \,.
\eeq
}
\end{thm}

\noindent The proof is in Section \ref{sectionlimits}.  
The first few rational functions $H_{k,d}(\lambda_1,\dots,\lambda_k,q)$ are listed here:
\begin{align}
& H_{2,1} \= \e^2 \frac{1}{(\lambda_1^2-\frac{\e^2}4)(\lambda_2^2-\frac{\e^2}4)} \,, \quad H_{2,2} \= \e^2 \frac{3\lambda_1^2+2\lambda_1\lambda_2+3\lambda_2^2-9\e^2}
{(\lambda_1^2-\frac{\e^2}4)(\lambda_2^2-\frac{\e^2}4)(\lambda_1^2-\frac{9\e^2}4)(\lambda_2^2-\frac{9\e^2}4)}\,, \nn\\
& H_{2,3} \= \e^2 \frac{10 \lambda_1^4+8 \lambda_1^3 \lambda_2+12 \lambda_1^2 \lambda_2^2+8 \lambda_1 \lambda_2^3+10 \lambda_2^4 - \e^2(110 \lambda_1^2+68 \lambda_1 \lambda_2+110  \lambda_2^2)+ 325 \e^4}
{(\lambda_1^2-\frac{\e^2}4)(\lambda_2^2-\frac{\e^2}4)(\lambda_1^2-\frac{9\e^2}4)(\lambda_2^2-\frac{9\e^2}4)(\lambda_1^2-\frac{25\e^2}4)(\lambda_2^2-\frac{25\e^2}4)} \,,\nn\\
& H_{3,1} \= \e^4 \frac{1}{(\lambda_1^2-\frac{\e^2}4)(\lambda_2^2-\frac{\e^2}4)(\lambda_3^2-\frac{\e^2}4)} \,, \nn\\
& H_{3,2} \= 
\text{\footnotesize{$\e^4 \frac{48 (\lambda_1^2 \lambda_2^2 + \lambda_1^2 \lambda_3^2 +\lambda_2^2\lambda_3^2) +32 \lambda_1 \lambda_2 \lambda_3(\lambda_1+\lambda_2+\lambda_3) 
- 24\e^2 ( \lambda_1 \lambda_2+ \lambda_1 \lambda_3 +\lambda_2\lambda_3 + 6 (\lambda_1^2+\lambda_2^2+\lambda_3^2)) + 351 \e^4} {(\lambda_1^2-\frac{\e^2}4)(\lambda_2^2-\frac{\e^2}4)(\lambda_3^2-\frac{\e^2}4)(\lambda_1^2-\frac{9\e^2}4)(\lambda_2^2-\frac{9\e^2}4)(\lambda_3^2-\frac{9\e^2}4)} $ } }.\nn
\end{align}

Finally we look at the $q\rightarrow \infty$ limit.

\begin{thm}  \label{qinflimit} 
\textit{
Fix $k\geq 2$. As $q\rightarrow \infty$, with fixed $\e$, $\lambda_1$,\dots,$\lambda_k$ and $\bigl|{\rm arg}\, (q^{1/2}/\e) \bigr|<\pi$,
\begin{align}
& \biggl(\prod_{i=1}^k \cos \frac{\pi \lambda_i}\e  \biggr)  \, H_k\Bigl(\frac{\lambda_1}\e, \dots, \frac{\lambda_k}\e; \frac{q^{1/2}}\e\Bigr) 
 \; \sim \;  \sum_{d\geq 0} q^{-\frac d2} \, H_k^{d,0}(\lambda_1,\dots,\lambda_k,\e)  \nn\\
&  ~~  \+  \sum_{d\geq 1} q^{-\frac d2} \, \sum_{m= 1}^{{\rm min} \{d,k\}} \bigg[H_k^{d,m}(\lambda_1,\dots,\lambda_k,\e) \cos \Bigl(4 m \frac{q^{1/2}}\e\Bigr)  
+ \widetilde H_k^{d,m}(\lambda_1,\dots,\lambda_k,\e) \sin \Bigl(4 m \frac{q^{1/2}}\e\Bigr) \bigg] \label{asympHq}
\end{align}
where $H_{k}^{d,m}(\lambda_1,\dots,\lambda_k,\e)$ and $\widetilde H_k^{d,m}(\lambda_1,\dots,\lambda_k,\e)$ 
are elements in the ring 
$$\QQ(\lambda_1,\dots,\lambda_k,\e)\biggl[\sin \frac{\pi \lambda_1} \e, \,
\cos\frac{\pi \lambda_1}\e, \, \dots, \, \sin\frac{\pi \lambda_k}\e, \, \cos \frac{\pi \lambda_k}\e \biggr]\,.$$}
\textit{For $k=1$ and $\bigl|{\rm arg}\, (q^{1/2}/\e) \bigr|<\pi$, as $q\rightarrow \infty$ (with fixed $\e,\lambda$),  the following asymptotic holds
\begin{align}
& \cos \Bigl(\frac{\lambda}\e \pi \Bigr) \,  H_1^*\Bigl(\frac{\lambda}\e; \frac{q^{1/2}}\e\Bigr) \,  \sim \,  
-\frac{\pi}{2} \sin \Bigl(\frac{\pi  \lambda }{\e}\Bigr) \+ \sin \Bigl(\frac{\pi  \lambda }{\e}\Bigr) \sum_{d\geq 0} \frac{(2d-1)!! \, \prod_{j=-d}^{d} (\lambda-\e j)}{(2d+1) \, d!\, 2^{3d+1} \, q^{d+\frac12}}   \nn\\
&\qquad\qquad\qquad\qquad\qquad\quad  \+ \cos \Bigl(4\frac{q^{1/2}}\e\Bigr)   \sum_{d\geq 1} q^{- \frac{d}2} H_1^{*\,d} 
\+ \sin \Bigl(4\frac{q^{1/2}}\e\Bigr)  \sum_{d\geq 1} q^{- \frac{d}2}  \widetilde H_1^{*\,d}\,  \label{asympHq1}
\end{align}
where $\e^{d-2} H_1^{*\,d}$, $ \e^{d-2} \widetilde H_1^{*\,d}$  
are elements in $ \QQ[\lambda,\e]$, and ${\rm gr} \, H_1^{*\,d} \= d \, H_1^{*\,d}$, ${\rm gr} \, \widetilde H_1^{*\,d} \= d \, \widetilde H_1^{*\,d}$. 
}
\end{thm}

\end{subtheorem}

\subsection{Organization of the paper} In Section \ref{sec2} we review the matrix resolvent approach and prove Theorem \ref{prop1}. In Section \ref{sec3} we prove 
Proposition \ref{ptopod}, Theorems \ref{pbispec}--\ref{thm05}. In Section \ref{sectionana} we prove 
Propositions \ref{nopoleH}--\ref{dhyper} and Theorems \ref{newformulaana}--\ref{thm2}. In Section \ref{sectionlimits} 
we  prove Theorems \ref{e0limit}--\ref{qinflimit}. 
Further remarks are in Section \ref{section5}. 

\subsection{Acknowledgements} 
We would like to thank Anton Mellit and Mattia Cafasso for discussions. 
One of the authors D.Y. is grateful to Youjin Zhang for his advice.  

\section{Matrix resolvent approach to the Toda Lattice Hierarchy} \label{sec2}
The matrix resolvent approach for computing tau-functions of the Toda Lattice Hierarchy was developed in \cite{DY1}.  
Let us give a short review. 
Let $L$ denote the following difference operator 
\beq\label{difference}
L \= \Delta \+ v_n \+ w_n \, \Delta^{-1} \,
\eeq
where $\Delta$ denotes the shift operator, i.e. $\Delta: \psi_n \mapsto \psi_{n+1}$.  The Toda Lattice Hierarchy is defined by
\begin{align}
& \frac{\p L}{\p t_i} \= \frac{1}{(i+1)!}\,  \bigl[A_i \,,\, L\bigr],\qquad i\geq 0\,, \label{TLH1}\\
& A_i \= \bigl(L^{i+1}\bigr)_+\,. \label{TLH2}
\end{align}
One observes that the normalization used here for the time variables $t_0,t_1,t_2,\dots$ 
is not the standard one \cite{DY1}, but the one suitable for the study of GW invariants of $\mathbb{P}^1$. 

\subsection{Toda conjecture}
We begin with a brief recall of the {\it Toda conjecture}, now a theorem:

\noindent {\bf Theorem} (\cite{OP}, \cite{DZ-toda}).   \textit{Denote 
$\F^{\rm s} \= \F^{\rm s}(x,\bt;\e) \:= \F \bigl(T^1_j = x \, \delta_{j,0}, \, T^2_j= t_j,\,j=0,1,\dots; \e;q=1 \bigr)$
with $\bt=(t_0,t_1,t_2,\dots)$. Let $Z:=e^{\F^{\rm s}}$. Define $u,v$ by
\begin{align}
& v \= v(x,\bt;\e) \:=  \e \frac{\p }{\p t_0 } \log \frac{Z(x+\e, \bt; \e)}{Z(x,\bt;\e)} \,, \label{defv}\\
& u \= u(x,\bt;\e) \:=  \log \frac{Z(x+\e, \bt; \e) \, Z (x-\e, \bt;\e)}{Z^2(x,\bt;\e)} \,. \label{defu}
\end{align}
Then $u,v$ satisfy the Toda Lattice Hierarchy with the first equation being}
\begin{align}
& \frac{\p v(x,\bt;\e)}{\p t_0} \= \frac1{\e} \, \Bigl(e^{u(x+\e,\bt;\e)} - e^{u(x,\bt;\e)}\Bigr) \, ,\\
& \frac{\p u(x,\bt;\e)}{\p t_0} \= \frac1{\e} \, \bigl(v(x,\bt;\e)-v(x-\e,\bt;\e)\bigr) \, .
\end{align}
The Toda conjecture was formulated in \cite{Du1, EY, Ge}, and was later proved by 
Okounkov--Pandharipande \cite{OP}; an extension \cite{DZ-toda} of this conjecture 
to the full generating function \eqref{freeenergy} requires an introduction of the 
extended Toda hierarchy \cite{CDZ} in terms of a suitably defined logarithm of the 
difference operator $L$ \eqref{difference}; see also \cite{OPeq}.
A slightly stronger version of this conjecture was also confirmed in the above proofs, 
namely, $Z$ is a particular {\bf tau-function} (in the sense of \cite{DZ-toda,DZ-norm,DY1}) 
of the Toda Lattice hierarchy. This property along with the {\it string equation} 
\beq\label{string}
\sum_{i=1}^\infty t_i \frac{\p Z}{\p t_{i-1}}  \+ \frac{x \, t_0}{\e^2} Z \= \frac{\p Z}{\p x}
\eeq
uniquely determines $Z$ up to a constant factor (independent of $\e$\,!) only. 

\subsection{Matrix resolvent}
Denote by $\mathbb{Z}[\bv,\bw]$ the ring of polynomials with integer coefficients 
in the infinite set of variables $\bv=(v_n), \, \bw=(w_n), \, n\in \mathbb{Z}$.
The (basic) matrix resolvent $R_n(\lambda)$ associated with $L$ is defined as
the unique solution to the following problem \cite{DY1}
\begin{align}
& R_{n+1}(\lambda) \, U_n(\lambda) \,-\, U_n(\lambda) R_n(\lambda) \= 0\,, \label{rg1} \\
& {\rm tr} \, R_n(\lambda) \= 1,\qquad {\rm \det} \, R_n(\lambda) \= 0\,,  \label{rg2} \\
& R_n(\lambda) \= \begin{pmatrix} 1 & 0  \\  0 & 0 \\ \end{pmatrix}
\+ {\rm O}\bigl(\lambda^{-1}\bigr) \, \in \, {\rm Mat} \bigl(2, \mathbb{Z}[\bv,\bw][[\lambda^{-1}]]\bigr)  \label{rg3}
\end{align}
where $U_n(\lambda) := \begin{pmatrix} v_n-\lambda & w_n\\ -1 & 0\end{pmatrix}$.
Write 
$$
R_n(\lambda) \= 
\begin{pmatrix} 1+\alpha_n(\lambda) & \beta_n(\lambda)  \\  \gamma_n(\lambda) & -\alpha_n(\lambda) \\ \end{pmatrix} \,, \quad 
\alpha_n(\lambda),\,\beta_n(\lambda),\, \gamma_n(\lambda)\in {\rm O}(\lambda^{-1})\,.
$$
Then the above equations \eqref{rg1}--\eqref{rg2} become a series of recursive relations for $\alpha_n,\beta_n,\gamma_n$:
\begin{align}
& \beta_n \= - w_n \, \gamma_{n+1} \,, \label{rr1} \\
& \alpha_{n+1} \+ \alpha_n \+ 1 \= (\lambda- v_n) \, \gamma_{n+1} \,, \label{rr2} \\
& (\lambda - v_n)  \, ( \alpha_n - \alpha_{n+1}) \= w_n \, \gamma_n \,-\, w_{n+1} \, \gamma_{n+2} \,, \label{rr3} \\
& \alpha_n \+ \alpha_n^2 \+ \beta_n \, \gamma_n  \= 0\,. \label{rr4} 
\end{align}
Along with the initial values \eqref{rg3} one can find $\alpha_n,\beta_n,\gamma_n$ in an algebraic way \cite{DY1}.  Indeed,
write 
$$
\gamma_n  \= \sum_{j\geq 0} \frac{c_{n,j}}{\lambda^{j+1}},\qquad \alpha_n \= \sum_{j\geq 0} \frac{a_{n,j}}{\lambda^{j+1}}\,.
$$
Substituting these expressions into \eqref{rr1}--\eqref{rr4} we obtain
\begin{align}
& c_{n,j+1} \= v_{n-1} \, c_{n,j} \+ a_{n,j} \+ a_{n-1,j} \, , \label{rec1}\\
& a_{n,j+1} \,-\, a_{n+1,j+1} \+ v_n \, [ a_{n+1,j} - a_{n,j}]  \+ w_{n+1} \,c_{n+2,j} \,-\, w_n \, c_{n,j} \= 0\,, \label{rec2}  \\
& a_{n,\ell} \=\sum_{i+j=\ell-1} \bigl[w_n \, c_{n,i} \, c_{n+1,j} - a_{n,i} \, a_{n,j}  \bigr]\,, \label{rec3} \\
& a_{n,0}=0, \qquad c_{n,0}=1 \,.\label{rec4}
\end{align}
We shall call \eqref{rg1}--\eqref{rg3}, or \eqref{rr1}--\eqref{rr4}, or \eqref{rec1}--\eqref{rec4} the {\it matrix-resolvent recursive
relations for the Toda Lattice Hierarchy}. It should be noted that these recursive relations are valid for an 
arbitrary solution $\bigl(v_n(\bt),w_n(\bt)\bigr)$ to the Toda Lattice Hierarchy; moreover, the form of matrix-resolvent recursive relations as well as equation \eqref{rec4} 
do not depend on the solution. 

\subsection{From matrix resolvent to tau-function}
For an arbitrary solution $\bigl(v_n(\bt),w_n(\bt)\bigr)$ of the Toda Lattice Hierarchy,  
there exists a unique (up to multiplying by exponential of an arbitrary linear function in $n,t_0,t_1,\dots$) 
function $\tau_n({\bf t})$ satisfying \cite{DY1}
\begin{align}
& \sum_{i,\, j\geq 0} \frac{\partial^2\log\tau_n({\bf t})}{\partial t_i \,\partial t_j} \frac{(i+1)! (j+1)!}{\lambda^{i+2} \mu^{j+2}} \= \frac{{\rm tr}\, R_n(\lambda; {\bf t})R_n(\mu;{\bf t})-1}{(\lambda-\mu)^2} \, ,
\label{taun1} \\
& \frac1{\lambda} \+ \sum\limits_{i\geq 0} \frac{\pal}{\pal t_i} \log \frac{\tau_{n+1}({\bf t})}{\tau_n({\bf t})}  \frac{(i+1)!}{\lambda^{i+2}} \=  \gamma_{n+1}(\lambda; {\bf t}) \, ,
\label{taun2} \\
&\frac{\tau_{n+1}({\bf t}) \tau_{n-1}({\bf t})}{\tau_n^2({\bf t})} \= w_n \, . \label{taun3}
\end{align}
Here, $R_n(\lambda;\bt):=R_n(\lambda)|_{v_n=v_n(\bt),\,w_n=w_n(\bt)}$. 
We call $\tau_n(\bt)$ the {\it tau-function} of the solution $\bigl(v_n(\bt),w_n(\bt)\bigr)$. Indeed, 
by interpolating using $x=n\epsilon$, we know that the Toda Lattice Hierarchy is a tau-symmetric integrable system of Hamiltonian 
PDEs within the normal form of \cite{DZ-norm}, and the identification between $\tau_n(\bt)$ and the 
tau-function of \cite{DZ-norm,CDZ,DZ-toda} is made in \cite{DY1}. 
By a straightforward residue computation (comparing coefficients of $\mu^{-2}$ in \eqref{taun1}), we obtain
\beq
\sum_{i\geq 0} \frac{(i+1)!}{\lambda^{i+2}} \frac{\partial^2\log\tau_n({\bf t})}{\partial t_0 \partial t_i} \= \alpha_n(\lambda;\bt)\,.
\eeq

\noindent {\bf Theorem A} (\cite{DY1}).  {\it Generating series of logarithmic derivatives of $\tau_n(\bt)$ have the following 
expressions}
\beq\label{thm-kp}
    \sum_{i_1,\dots,i_k\geq 0} \frac{\partial^k\log\tau_n({\bf t})}{\partial t_{i_1} \cdots \partial t_{i_k}}\prod_{\ell=1}^n \frac{(i_\ell+1)!}{\lambda_\ell^{i_l+2}} 
 \=   - \sum_{\sigma\in S_k/C_k} \frac{\tr \,\left[ R_n(\lambda_{\sigma_1};\bt)\cdots R_n(\lambda_{\sigma_k};\bt)\right]}
{\prod_{i=1}^k (\lambda_{\sigma_i}-\lambda_{\sigma_{i+1}})}\, , \quad \forall\,k\geq 3\,.
\eeq

\subsection{Proof of Theorem \ref{prop1}} \label{section3}
The first step is to give the initial value of the GW solution. 
\begin{lemma}\label{simplelemma} 
The initial value of the solution \eqref{defv}, \eqref{defu} of the Toda Lattice Hierarchy is given by 
\begin{align}
& u(x,\bt= {\bf 0}; \e) \= 0\,, \nn\\
& v(x,\bt= {\bf 0}; \e) \= x+ \frac \e 2 \, . \nn
\end{align}
\end{lemma}
\pf 
The string equation \eqref{string} can be written equivalently as
\beq\label{stringF}
\sum_{i=1}^\infty t_i \frac{\p \F^{\rm s}}{\p t_{i-1}}  \+ \frac{x \, t_0}{\e^2}  \= \frac{\p \F^{\rm s}}{\p x}\,.
\eeq
Differentiating both sides of \eqref{stringF} w.r.t. $t_0$ we obtain 
$\sum_{i=1}^\infty t_i \frac{\p^2 \F^{\rm s}}{\p t_{i-1} \p t_0}  + \frac{x}{\e^2}  = \frac{\p^2 \F^{\rm s}}{\p x \p t_0}$.
Taking $t_1=t_2=\dots = 0$ in this equation gives
$\frac{\p^2 \F^{\rm s}}{\p x\, \p t_0}(x,t_0,0,0,\dots;\e) \equiv \frac{x}{\e^2}$.
In particular, $\frac{\p^2 \F^{\rm s}}{\p x \p t_0} (x,0,0,\dots;\e) = \frac{x}{\e^2}$.
Therefore using \eqref{defv} we have
$$
v(x,0,0,\dots;\e) \= \frac{e^{\e \p_x} -1}{\e \p_x} \e^2   \frac{\p^2 \F^{\rm s}}{\p x \p t_0} (x,0,0,\dots;\e) \=  x \+ \frac \e 2 \,.
$$
We now look at the initial value of $u$. Since 
$$
u (x,0,0,\dots;\e) \= \F^{\rm s} (x+ \e, 0,0, \dots;\e ) \+ \F^{\rm s}(x-\e, 0,0, \dots;\e) \,-\, 2 \F^{\rm s}(x, 0, 0, \dots;\e)\,,
$$
we only need to find coefficients of $x^n$ in the Taylor expansion of $\F^{\rm s}(x,0,0,\dots; \epsilon)$.
The degree-dimension matching implies  
$2g -  2 + 2d + n = 0$.
So the only possible choices are $(g,d,n)=(0, 0, 2), (0, 1, 0), (1, 0,0)$. The constant terms do not contribute to $u (x,0,0,\dots;\e)$. 
The quadratic term cannot appear because of the well-known expression of the genus zero primary potential (the potential of the corresponding Frobenius manifold) is 
$$
F\= \bigl(\epsilon^2{\mathcal F}^{\rm s}\bigr)_{
\epsilon=0} \= \frac12 (v^1)^2 v^2 \+ e^{v^2}\,, \quad \mbox{with} ~ v^1=v|_{\epsilon=0}, ~ v^2=u|_{\epsilon=0}.
$$
Clearly, after restricting to $v^2=0$ and $v^1=x$, there is no $x^2$ term.
\epf

We now proceed to the proof of Theorem \ref{prop1}.
Recall the interpolation formula $ x = n \, \epsilon$. Then the above Lemma~\ref{simplelemma} implies that, 
for the particular solution \eqref{defv}--\eqref{defu} to the Toda Lattice Hierarchy
\begin{align}
& u_n(\bt= {\bf 0}; \e) \= 0\,, \label{un0}\\
& v_n(\bt= {\bf 0}; \e) \= \e \, n+ \frac \e 2 \, . \label{vn0}
\end{align}
Substituting \eqref{un0}--\eqref{vn0} into \eqref{rr1}--\eqref{rr4} 
we obtain the following recursive relations for the entries of the {\it initial} (basic) matrix resolvent 
\begin{align}
& \alpha_{n+1} \+ \alpha_n \+ 1 \=  \Bigl(\lambda - \e \, n- \frac \e 2\Bigr) \, \gamma_{n+1}  \,, \label{rrr1} \\
& \Bigl(\lambda - \e \, n- \frac \e 2 \Bigr)  \, ( \alpha_n - \alpha_{n+1}) \=  \gamma_n \,-\,   \gamma_{n+2} \,, \label{rrr2} \\
& \alpha_n \+ \alpha_n^2  \,-\,  \gamma_n \, \gamma_{n+1} \= 0\,. \label{rrr3} 
\end{align} 
Here, ``initial" means at ${\bt}=0$. 
The theorem is then proved by taking $\bt=0$ in \eqref{thm-kp} and \eqref{taun1}. \epf

As before, write 
\beq\label{gaform}
\gamma_n \= \gamma_n(\lambda;\e)  \= \sum_{j\geq 0} \frac{c_{n,j}}{\lambda^{j+1}},\qquad \alpha_n \= \alpha_n(\lambda;\e) \= \sum_{j\geq 0} \frac{a_{n,j}}{\lambda^{j+1}}\,.
\eeq
Then equations \eqref{rrr1}--\eqref{rrr3} become
\begin{align}
& c_{n,j} \= \e \Bigl(n - \frac 1 2 \Bigr) \, c_{n,j-1} + a_{n,j-1} + a_{n-1,j-1}  \,, \label{a1} \\
& a_{n,j} \,-\, a_{n+1,j} \+ \Bigl(\e \, n+ \frac \e 2\Bigr) \, ( a_{n+1,j-1} - a_{n,j-1})  \+  \, c_{n+2,j-1} \,-\,  c_{n,j-1} \= 0\,, \label{a12} \\
& a_{n,j} \= \sum_{i=0}^{j-1} \bigl(c_{n,i} \, c_{n+1,j-1-i} - a_{n,i} \, a_{n,j-1-i} \bigr) \label{a2} 
\end{align}
together with the initial data for the recursion  
\beq\label{gaini}
a_{n,0} \=0, \quad c_{n,0} \=1\,.
\eeq

The first several terms of $\alpha_n,\gamma_n$ are given by
\begin{align}
& \alpha_n \= \frac{1}{\lambda ^2} \+ \frac{2 n \e}{\lambda ^3} \+ \frac{3  n^2 \e^2 +\frac{\e^2}{4}+3}{\lambda ^4} \+ \frac{4 n^3 \e^3 +  n (\e^3 +12 \e )}{\lambda ^5} \+ \cdots\,, \label{first1}\\
& \gamma_n \=  \frac{1}{\lambda } \+ \frac{n\e-\frac{\e}{2}}{\lambda ^2}  \+ \frac{n^2 \e^2  - n\e^2 +\frac{\e^2}{4}+2}{\lambda ^3} 
\+ \frac{ n^3\e^3 - \frac32 n^2 \e^3 + n (\frac{3 \e^3 }{4} +6 \e )-\frac{\e^3}{8}-3 \e}{\lambda ^4} \+ \cdots. \label{first2}
\end{align}

\section{Solving the matrix-resolvent recursive relations of $\mathbb{P}^1$} \label{sec3}
The goal of this section is to solve equations \eqref{r1}--\eqref{r3}. 
We start with proving Proposition \ref{ptopod}.
\subsection{Proof of Proposition \ref{ptopod}} Write
$M(z,s)=\begin{pmatrix} 1+a(z,s) & b(z,s)  \\  c(z,s) & -a(z,s) \\ \end{pmatrix}$.
The topological difference equation \eqref{m1xp}, i.e. 
$M(z-1;  s) \begin{pmatrix} z-1/2 & -s\\ s & 0\end{pmatrix} \=  \begin{pmatrix} z-1/2 & -s\\ s & 0\end{pmatrix} \, M(z;s)$,
written in terms of $a,b,c$ reads as follows:
\begin{align}
& b(z)  \+  c(z-1) \= 0\,, \nn\\
& \bigl(z-\frac12 \bigr)\, b(z)  \+ s \, \bigl( 1+ a(z-1) +  a(z) \bigr) \=0\,, \nn \\
& s \, b(z-1) \+ s \, c(z)   \+ \bigl(z -\frac12\bigr) \, \bigl( a(z-1)  -  a(z) \bigr)  \=0\,. \nn
\end{align}
Here, $a(z),b(z),c(z)$ are short notations for $a(z,s),b(z,s),c(z,s)$; below we 
keep using these notations when no confusion will occur. It follows from these equations that
\beq\label{bca}
 c(z) \= s  \, \frac{1 + a(z) + a( z+1)}{z + \frac12} \,, \qquad  b(z) \= - s \, \frac{1 + a(z-1) + a(z)}{ z-\frac12} \,. 
\eeq
Moreover, the topological difference equation is reduced to the following 3rd order linear difference equation (with a parameter $s$) for $a$:
\beq\label{scalartopo}
s^2 \Bigl(\frac{ 1 +  a(z)+ a(z+1) }{ z+\frac12}-\frac{1+a(z-2)+a(z-1)}{z-\frac32}\Bigl) \+ \bigl(z -\frac12\bigr)\, \bigl(a(z-1) -  a(z) \bigr)  \=0\,.
\eeq
Write 
$a(z,s) \=  \sum_{k\geq 0} A_k \, z^{-k-1} $. Then equation \eqref{scalartopo} is equivalent to the following equations:
\begin{align}
& -8(k+2) A_{k+1}  \=   - \,16 \, s^2 \, \delta_{k,0}  \,-\,  3 \bigl(1+4 s^2\bigr) A_{k-1}  \+ 2 \bigl(1+4s^2\bigr)A_k \nn\\ 
& \qquad\qquad \+  \sum_{k_1,n\geq 0 \atop k_1+n+1=k}  A_{k_1} \bigl((-1)^{n+1} 12\, s^2- 2^{n+2} s^2 +(3-4s^2)\bigr) \binom{1+k_1}n  \nn\\
& \qquad\qquad \+ \sum_{k_1,n\geq 0\atop k_1+n =k} A_{k_1} \bigl( (-1)^n 8 \, s^2-2^{n+3} s^2  - 2(1+4s^2)\bigr) \binom{1+k_1}n \nn\\
& \qquad\qquad \, - \, 12 \sum_{k_1\geq 0,n\geq 1 \atop k_1+n=k+1}  A_{k_1} \binom{1+k_1}n \+ 8 \sum_{k_1\geq 0,n\geq 2 \atop k_1+n=k+2}  A_{k_1} \binom{1+k_1}n \,, \qquad k\geq -1\,. 
\end{align}
Here it is understood as $A_{-2}=A_{-1}=0$. Together with \eqref{bca}, this recursion proves the existence and uniqueness of a solution $M^*$ of the form 
$M^* = \begin{pmatrix} 1 & 0  \\  0 & 0 \\ \end{pmatrix}
+ \sum_{k\geq 1}  M^*_k \, z^{-k}$.  Moreover, the fact that each entry of $M^*_k$ is a polynomial of $s$ can be seen easily from this recursion (for $A_k$). 
Finally, taking the determinants of both sides of \eqref{m1xp} we have
$$
s^2 \, \det M^*(z-1) \= s^2 \, \det M^*(z)  \quad \Rightarrow \quad  \det M^*(z-1) \= \det M^*(z) \,. 
$$
It is easy to see that $\det M^*(z) \in z^{-1}Q[s][[z^{-1}]]$. Noting that $A_0=0$ we find that the coefficient of $z^{-1}$ in $\det M^*(z)$ also vanishes. Therefore, $\det M^*(z)$ vanishes. 
The proposition is proved. \epf

\subsection{Proof of Theorem \ref{pbispec}} In this subsection, we prove Theorem \ref{pbispec}. 
The proof is similar with the one given in \cite{BDY3} 
(see the ``Key Lemma" i.e. Lemma 4.2.3 therein). 

\noindent {{\it Proof} of Theorem \ref{pbispec}}.   
Define $R_n^*(\lambda;\e)=M^*\bigl(\frac{\lambda}\e-n; \frac1\e\bigr)$. It is easy to check that 
$R_n^*(\lambda;\e)$ satisfies \eqref{r1}--\eqref{r3}. Since 
$R_n(\lambda;\e)$ is the unique solution to \eqref{r1}--\eqref{r3}, we have $R_n(\lambda;\e)=R_n^*(\lambda;\e)$. 
By definition, $\R(\lambda;x;\e)=R_{x/\epsilon}(\lambda;\epsilon)$. Hence $\R(\lambda;x;\e)$ only depends on  
$\lambda-x$ and $\e$.
The theorem is proved. \epf

\subsection{Proof of Theorem~\ref{conj-corr} and Theorem~\ref{thm05}} 
To prove Theorems~\ref{conj-corr} and~\ref{thm05}, we must show two things:
\begin{itemize}
\item[i)] Prove that the entries of the matrix-valued meromorphic function $B(z;s)$ defined by equations \eqref{defBzs}--\eqref{Hyper2}
have asymptotic expansions as power series in $z^{-1}$ (for $|z|\rightarrow \infty$ at a bounded distance away from half integers) given by the RHS of~\eqref{ram}.
\item[ii)] Show that the function~$B(z;s)$ satisfies the properties~\eqref{m1xp}--\eqref{m1xp1} with $M$ replaced by~$B$. 
\end{itemize} 

For step~i),  we must look at the asymptotics of $G(z;s)$ and $\widetilde G(z;s)$ as $|z|\rightarrow \infty$ with $s$ bounded, say $|s|\leq S$.  We 
consider only the case of~$G$, since the case of $\widetilde G$ is exactly similar. We claim first that 
\begin{align}
G(z;s) \:=  \sum_{m=0}^\infty \binom{2m}m \,  \frac{s^{2m}}{\bigl(z - m +  \frac12\bigr)_{2m}}  
\= \sum_{m=0}^{N-1} \binom{2m}m \,  \frac{s^{2m}}{\bigl(z - m +  \frac12\bigr)_{2m}} \+ {\rm O}\bigl(z^{-2N}\bigr) \label{star}
\end{align}
for any fixed $N\in\NN$ as $z\rightarrow \infty$ at a bounded distance from $\ZZ+\frac12$. Indeed, the terms with 
$N\leq m\leq \frac12|z|$ in~\eqref{star} 
are individually bounded by~$\frac{2^{2m} S^{2m}} {(|z|/2)^{2m}}$ (because each factor in the Pochhammer symbol in the denominator has absolute value $\geq\frac12 |z|$), 
so their sum is $\leq \sum_{m= N}^\infty \bigl(\frac{4S}{|z|}\bigr)^{2m} = {\rm O}\bigl(z^{-2N}\bigr)$. The terms with $m>\frac12|z|$ are individually bounded by 
~$\frac{2^{2m}s^{2m}}{\delta \, m! \, (m-1)!}$, where $\delta$ is the distance from $z$ to $\ZZ+\frac12$, so their sum is smaller than any fixed negative powers of $|z|$ as 
$|z|\rightarrow \infty$ with $\delta$ fixed.  Now using a partial fraction development in each summand in~\eqref{star}, we find
\begin{align}
G(z;s) & \= 1\+2\sum_{m=1}^{N-1} \frac{s^{2m}}{m! \, (m-1)!} \sum_{\ell=0}^{2m-1} \frac{(-1)^{\ell} \binom{2m-1}\ell}{z-m+\ell+\frac12} \+ {\rm O}\bigl(z^{-2N}\bigr)\,\nn\\
& \= 1 \+ 2 \sum_{r=1}^{2N-1} \frac1{z^r}\,  \sum_{0\le\ell<2m\le r} \frac{(-1)^\ell s^{2m}}{m! (m-1)! }\, \binom{2m-1}\ell \bigl(m-\ell-\tfrac12\bigr)^{r-1}  
\+ {\rm O} \bigl(z^{-2N}\bigr)\,,
\end{align}
where we have removed the terms with $2m>r$ because the $(2m-1)$st (backwards) difference of a polynomial of degree~$r-1$ vanishes identically if~$2m>r$.
We also note that the terms with $r$ odd give zero (replace $\ell$ by $2m-1-\ell$), so we can set $r=2j+2$, $m=i+1$ to recover the expression given in~\eqref{defofalpha},
proving that $G(z;s)\sim1+2\alpha$ as claimed.

Now we do step~ii). The explicit expression for $M$ given in the statement of Theorem~\ref{conj-corr} 
clearly has the form $M=\begin{pmatrix} 1 & 0  \\  0 & 0 \\ \end{pmatrix}
+ \sum_{k\geq 1}  M_k \, z^{-k}$. Therefore using Proposition \ref{ptopod} we only need to show that 
$M(z,s)$ satisfies \eqref{m1xp}--\eqref{m1xp1}. Then due to Proposition \ref{thm05} it suffices to show that 
$B(z,s)$ satisfies \eqref{m1xp}--\eqref{m1xp1} for $z\in\CC-\ZZ_{\rm odd}$. Identity \eqref{m1xp1} is obvious for $B(z,s)$. Identity \eqref{m1xp} is equivalent to 
\begin{align}
& \widetilde G(z;s) \=  \frac{G(z;s)+ G(z+1,s)}{2} \, , \label{1sum2} \\
& \frac{\widetilde G(z+\frac12,s)}{z+1} - \frac{\widetilde G(z-\frac32,s) }{z-1} \= \frac{z}{2s^2} \biggl[G\bigl(z+\frac12,s\bigr) - G\bigl(z-\frac12,s\bigr)\biggr]\,. \label{1diff} 
\end{align}
Identity \eqref{1sum2} is true since 
$G(z;s) \+ G(z+1,s) =  \sum_{i=0}^\infty \binom{2i}{i} \Bigl[  \frac{s^{2i}}{(z-i+\frac12)_{2i}} +  \frac{s^{2i}}{(z-i+\frac32)_{2i}} \Bigr]=2 \, \widetilde G(z;s)$. 
Similarly, we find that identity \eqref{1diff} is true. 
\epf

Note that the $k$-point function $F_k(\lambda_1, \dots, \lambda_k; 0; \e,1)$ ($k\ge2$) can be expressed in terms of $M$ by 
\begin{align}
& F_k(\lambda_1, \dots, \lambda_k;x;\e,1) \= -\sum_{\sigma\in S_k/C_k} \frac{ \tr \,\Bigl[M\Bigl(\frac{\lambda_{\sigma(1)}-x}{\e};\frac1\e\Bigr)
\cdots M\Bigl(\frac{\lambda_{\sigma(k)}-x}\e;\frac1\e\Bigr)\Bigr]} {\prod_{i=1}^k \bigl(\lambda_{\sigma(i)}-\lambda_{\sigma(i+1)}\bigr)} 
\, - \, \frac{\delta_{k,2}}{(\lambda_1-\lambda_2)^2}\,\,.\label{npoint00}
\end{align}
\noindent The validity of this identity is understood in the formal power series ring $\QQ\bllm x,\lambda_1^{-1},\dots,\lambda_k^{-1}\brrm$.

\begin{prop}
For any $k\geq 2$, the following formula holds true
\begin{align}
& \e^k \, \sum_{i_1,\dots,i_k\geq 0}  
\frac{(i_1+1)! \cdots (i_k+1)!}{\lambda_1^{i_1+2} \cdots \lambda_k^{i_k+2}} 
 \sum_{m,g,d\geq 0\atop  2g+2d-2+2m=  \sum i_{\ell}}   \frac{q^{d}}{m!}  \e^{2g-2} \, \bigl\langle \tau_{i_1}(\omega) \cdots \tau_{i_k}(\omega) \tau_0(1)^m\bigr\rangle_{g,d} \nn\\
&\= -\frac{1}{k}\sum_{\sigma\in S_k} 
\frac{\tr \,\Bigl[{\mathcal R}\bigl( q^{-1/2} \lambda_{\sigma_1}; q^{-1}; q^{-1/2} \e\bigr)\dots {\mathcal R}\bigl( q^{-1/2} \lambda_{\sigma_k}; q^{-1} ; q^{-1/2} \e\bigr)\Bigr]}
{\prod_{i=1}^k (\lambda_{\sigma_i}-\lambda_{\sigma_{i+1}})} \,-\,  \frac{\delta_{k,2}}{(\lambda_1-\lambda_2)^2} \,.
\end{align}
\end{prop}
\noindent {\it Proof}. Use Proposition \ref{prop1}, \eqref{RDon}, as well as \eqref{scaling2}. \epf

\begin{remark}
R.~Pandharipande~\cite{P} proves that the numbers $\bigl\langle\tau_1(\omega)^{2g-2+d}\bigr\rangle_{g,d}$ coincide with the 
classical Hurwitz numbers $H_{g,d}$ defined by Hurwitz~\cite{Hurwitz}. A polynomial algorithm of 
computing these numbers has been obtained very recently~\cite{DYZ1} based on Pandharipande's equation \cite{P,DYZ1}. 
Although the formula~\eqref{npoint00} for $F_k$ with $k=2g-2+d$ contains the numbers
$H_{g,d}$, the algorithm designed from~\eqref{npoint00} is not of polynomial-time (note that 
however \eqref{npoint00} contains much more information than $H_{g,d}$). 
\end{remark}

\subsection{Proof of Proposition \ref{nopoleH}} The main observation is that for any $k\geq 3$, we have
\begin{align}
H_k (z_1,\dots,z_k) 
& \= - \sum_{\sigma\in S_{k}/C_{k}}  \frac{\tr \,\left[B(z_{\sigma(1)}, s) \dots B(z_{\sigma(k)},s)\right]}
{\prod_{i=1}^k (z_{\sigma_i}-z_{\sigma_{i+1}})} \nn\\ 
& \=  - \sum_{j=1}^{k-1} \sum_{\sigma\in S_{k-1}/C_{k-1}} \frac{\tr \,\left[B(z_{\sigma(1)}, s) \dots [B(z_k,s), B(z_{\sigma(j)},s)]\dots B(z_{\sigma(k-1)},s)\right]}
{(z_k-z_{\sigma(j)}) \prod_{i=1}^{k-1} (z_{\sigma_i}-z_{\sigma_{i+1}})} \,. \nn
\end{align}
So $H_k$ is analytic along $z_k = z_i$ for $i\neq k$ away from the half-integer points. Note that $H_k(z_1,\dots,z_k)$ is totally symmetric w.r.t. permutations of $z_1,\dots,z_k$. Therefore 
$H_k$ is also analytic along $z_j=z_i$ for $i\neq j$ (for any $j$). 
The case $k=2$ follows immediately from the second equality in~\eqref{twopointholo}.
\epf

\section{Proof of the factorization formulas} \label{sectionana}
We begin by giving the proof of Proposition~\ref{BBB} of Section~\ref{sec1.3}, giving an explicit factorization of the rank~1 matrix $B(z,s)$ as the 
product of a column vector and a row vector.

\noindent {{\it Proof} of Proposition \ref{BBB}}.  We have to prove the following three identities for hypergeometric ${}_1F_2$-functions 
as sums of products of Bessel functions:
\begin{align}
& \frac{1+ G(z;s)}2  \= \frac{\pi s} {\cos (\pi z) }  J_{z-\frac12} (2s) \, J_{-z-\frac12}(2s) \,,  \label{GBB1} \\ 
& \frac{1-G(z;s)}2 \= \frac{\pi s} {\cos (\pi z) }   J_{z+\frac12}(2s) \, J_{-z+\frac12}(2s) \,, \label{GBB2} \\ 
&  \frac{s}{z+\frac12}\widetilde G(z;s) \= \frac{\pi s} {\cos (\pi z) }   J_{z+\frac12} (2s) \, J_{-z-\frac12}(2s)\,.  \label{GBB3}
\end{align}
Indeed, 
\begin{align}
\mbox{RHS of}~ \eqref{GBB3} & \= \frac{\pi s} {\cos (\pi z)} 
\frac{1}{\Gamma(z+\frac32)\Gamma(-z+\frac12)} \, \sum_{n\geq 0} (-1)^n \, \frac{s^{2n}}{n!^2 \, \binom{n+z+\frac12}{n}}  \, \sum_{n\geq 0} (-1)^n \, \frac{s^{2n}}{n!^2 \, \binom{n-z-\frac12}{n}}  \nn\\
& \= \frac{s}{z+\frac12} \, \sum_{n\geq 0} s^{2n}  \sum_{n_1+n_2=n}  \, \frac{(-1)^n}{(n_1!)^2 \, (n_2!)^2 \binom{n_1+z+\frac12}{n_1}  \binom{n_2-z-\frac12}{n_2}} \= \mbox{LHS of}~ \eqref{GBB3} \,. \nn
\end{align}
Similarly one proves \eqref{GBB1},\, \eqref{GBB2}.  The factorization $B = u(z) \, u(-z)^T$ can also be verified directly.
\epf

\noindent {\it Proof} of Theorem \ref{newformulaana}.  
For $k\ge2$ we have
\begin{align}
\tr \, (B(z_1) \dots B(z_k) ) & \=  \tr \, \Bigl(u(z_1) u(-z_1)^T u(z_2) u(-z_2)^T \dots u(z_k) \, u(-z_k)^T\Bigr) \nn\\
& \=  \tr \, \Bigl(u(-z_1)^T u(z_2) \dots u(-z_k)^T u(z_1) \Bigr)  \=  \prod_{i=1}^k \, \bigl( u(-z_i)  ,  u(z_{i+1})  \bigr) \,. \nn
\end{align} (indices modulo $k$). Hence each summand in the trace-product formulas \eqref{twopointholo}, \eqref{npointholo} has the form 
\beq\label{summand}
\frac{\tr \, \bigl(B(z_1,s) \dots B(z_k, s) \bigr) }{ \prod_{i=1}^k (z_i-z_{i+1})} 
\= \e^k \, \prod_{i=1}^k D\bigl(z_i, z_{i+1};s\bigr)    \nn
\eeq
where we recall that $D(a,b;s) = u(-a,s)^T \, u(b,s)/(a-b)$, as claimed. \epf

Formula~\eqref{Dhgm} implies the following asymptotic formula for $a,b\notin \ZZ+\frac12$, as $a,b\rightarrow \infty$:
\beq\label{asymD}
D(a,b;s) - \frac1{a-b}  \; \sim \;  \sum_{p,q\geq 0} \frac{(-1)^{q+1}}{a^{p+1} b^{q+1}} 
\sum_{n\geq 1} \frac{s^{2n}}{n!} \sum_{1\leq i,j\leq n\atop i+j\leq n-1} (-1)^{i+j}\frac{(i+j-2n)_{n-1} \bigl(i-\frac12\bigr)^p \bigl(j-\frac12\bigr)^q}{(i-1)!(j-1)!(n-i)!(n-j)!} \,.
\eeq
Proposition~\ref{nopoleH} can be alternatively proved by using this formula.

\noindent {\it Proof }of Theorem \ref{thm2}.~
Differentiating both sides of \eqref{stringF} w.r.t. $t_j$~($j\geq 1$) we obtain  
$$\frac{\p \F^{\rm s}}{\p t_{j-1}} \+ \sum_{i=1}^\infty t_i \frac{\p^2 \F^{\rm s}}{\p t_{i-1} \p t_j}  \= \frac{\p^2 \F^{\rm s}}{\p x \p t_j}\,.$$
Setting $t_0= t_1=\dots =0$ in this equation yields $\llll \tau_{j-1} (\omega) \rrrr(x;\e,1) \= \llll \tau_0 (1) \tau_j (\omega) \rrrr(x;\e,1)$.
Note that equation \eqref{taun2} implies 
$$
\frac1\lambda \+ \e^2 \sum_{i=0}^\infty \frac{(i+1)!}{\lambda^{i+2}} \frac{e^{\e \p_x} -1}{\e \p_x} \bll \tau_0 (1) \tau_i (\omega) \brr (x;\e,1) \= \gamma_{n+1}\,,\qquad x= n \epsilon.
$$
So 
$\frac1\lambda \+ \frac{\e^2}{\lambda^{2}} \Bigl(\frac x{\e^2} + \frac1{2\e}\Bigr) \+
 \e^2 \sum_{i=1}^\infty \frac{(i+1)!}{\lambda^{i+2}} \frac{e^{\e \p_x} -1}{\e \p_x} \bll \tau_{i-1} (\omega) \brr (x;\e,1) \= \gamma_{n+1}$.
Therefore,
$$
-\e \frac{e^{\epsilon \p_x} -1}{\e \p x} \frac{\p F_1 (\lambda;x;\e,1)}{\p \lambda} 
\= \gamma_{n+1} - \frac{\e^2}{\lambda^{2}} \Bigl(\frac x{\e^2} + \frac1{2\e}\Bigr) -\frac1\lambda  \, .
$$ 
Hence we have
\begin{align}
-\frac{\p F_1(\lambda;x;\e,1)}{\p \lambda} 
& \= \sum_{j\geq 2} \frac{1}{\lambda^{j+1}} \sum_{i=0}^\infty  \frac{\e^{-1-2i}}{i!^2}    \sum_{l=0}^{2i} (-1)^{l} \binom{2i}{l} \sum_{m=0}^j \binom{j}{m} B_m \, \e^m  \Bigl(  x + \e \,i +\frac\e2 - \e \, l \Bigr)^{j-m} \,.   \nn
\end{align}
Identity \eqref{onepoint} follows immediately.  The equivalence between this identity and the statement of the theorem has 
already been explained in Section~\ref{sec1.4}.
\epf

\section{Four asymptotics} \label{sectionlimits} 
We have studied several analytic properties of $H_k(z_1,\dots,z_k;s)$.  Motivated by 
the GW theory, in this section, we will investigate further the functions
$H_k\bigl(\frac{\lambda_1}\e, \dots, \frac{\lambda_k}\e; \frac{q^{1/2}}\e\bigr)$. It should be noted that the 
function $W(\lambda,\e,q):=B\bigl(\frac{\lambda}\e; \frac{\sqrt{q}}{\e}\bigr)$ satisfies the following set of equations:
\begin{align}
& W (\lambda+\e,\e, q) \begin{pmatrix} \lambda-\frac{\e}2 & -\sqrt{q}\\ \sqrt{q} & 0\end{pmatrix} 
\=  \begin{pmatrix} \lambda-\frac{\e}2 & - \sqrt{q}\\ \sqrt{q} & 0\end{pmatrix} \, W (\lambda,\e, q) \,, \label{topoB}  \\
& \tr \,  W (\lambda,\e, q) \= 1 \,, \qquad  \det \, W (\lambda,\e, q) \= 0\,, \label{trdetB}\\ 
& W (\lambda,\e, q) \= \begin{pmatrix} 1 & 0  \\  0 & 0 \\ \end{pmatrix}
\+ {\rm O}\bigl(\lambda^{-1}\bigr)  \,, \quad \lambda\rightarrow \infty \,. \label{asympB}
\end{align}

\subsection{The $\e\rightarrow 0$ asymptotic. Proof of Theorem \ref{e0limit}} First we consider $k\geq2$. 
\begin{lemma}  \label{lemmae0} For $0<\frac{2\sqrt{q}}\lambda <1$ and $\frac\lambda\e>0$, 
the following asymptotic formula holds true  as $|\e|\rightarrow 0\,:$
\begin{align}
&  \frac12G\Bigl(\frac{\lambda}\e;\frac{\sqrt{q}}\e\Bigr) -\frac12 \; \sim \; \sum_{m=0}^\infty \e^m a_m(\lambda;q)  \,, 
\nn\\
& \frac{\sqrt{q}}{\lambda+\frac\e2} \, \widetilde G\Bigl(\frac{\lambda}\e;\frac{\sqrt{q}}\e\Bigr) \; \sim \; \sum_{m=0}^\infty \e^m c_m(\lambda;q) \nn
\end{align}
where $a_m\,, c_m$ are algebraic functions of $\lambda,\sqrt{q}$. 
Moreover, for $k,m\geq 0$, the functions $a_{2k+1}(\lambda;q)$ vanish, and $a_{2k}$, $c_m$ satisfy the homogeneity conditions:
${\rm gr} \,  a_{2k} = - 2k \, a_{2k}$, $~{\rm gr} \, c_m = -m \, c_m$.
\end{lemma}
The first several $a_{2k},\,c_m$ are given explicitly by
\begin{align}
& a_0 \= \frac{\lambda}{2(\lambda^2-4q)^{\frac12}} - \frac12\,,\quad a_2 \=   \frac{q\lambda(\lambda^2+16q)}{4(\lambda^2-4q)^{\frac72}}\,, \quad a_4 \=   \frac{q\lambda(\lambda^6+247q\lambda^4+2848q^2\lambda^2+3072q^3)}{16(\lambda^2-4q)^{\frac{13}2}} \,, \nn\\
& c_0 \= \frac{\sqrt{q}}{(\lambda^2-4q)^{\frac12}}\,, \quad c_1 \= -\frac{\sqrt{q} \, \lambda}{2(\lambda^2-4q)^{\frac32}} \,, \quad c_2 \= \frac{\sqrt{q} \, (\lambda^4+6 q \lambda^2)}{4(\lambda^2-4q)^{\frac72}}\,, \quad c_3 \= \frac{\sqrt{q}\lambda(\lambda^4+42 q\lambda^2+96 q^2)}{8(\lambda^2-4q)^{\frac92}}  \,.\nn
\end{align}
Lemma \ref{lemmae0} can be proved either by studying the analytic functions \eqref{Hyper1}--\eqref{Hyper2}, or by the following 
 lemma regarding the large-order 
asymptotics
of the Bessel functions \cite{Watson}. 
\begin{lemma} \label{lemmae0bessel}
For any fixed valued $\zeta\in (0,1)$, the following asymptotic holds true: as $\nu\rightarrow +\infty$, 
\beq
J_{\nu-\frac12}(\nu\,\zeta )\;\sim\;  \frac{\bigl(\nu-\frac12\bigr)^{\nu-\frac12}}{\Gamma\bigl(\nu+\frac12\bigr)} e^{V} \,, \quad 
V \= \nu \, V_0 \+ V_1 \+ \frac{V_2}{\nu} + \frac{V_3}{\nu^2} \+ \dots \,, 
\eeq
where $V_m$, $m\geq 0$ are functions of $\zeta$  with the first few given by
\begin{align}
& V_0 \= -1+ \sqrt{1-\zeta^2} \+ \log \zeta - \log \bigl(1+\sqrt{1-\zeta^2}\bigr) \,, \nn\\
& V_1 \= \frac12 + \frac12 \log\bigl(1+\sqrt{1-\zeta^2}\bigr) - \frac12 \log \zeta -\frac14 \log(1-\zeta^2) \,,  \nn\\
& V_2 \=  -\frac{1}{6  } +\frac{1}{4}\frac1 {(1-\zeta^2)} - \frac5{24} \frac{1}{(1-\zeta^2)^{\frac32}} \,,  \nn\\
& V_3 \= -\frac{1}{48} +\frac{1}{4}\frac1{(1-\zeta^2)^{\frac32}} -\frac{1}{4} \frac1{(1-\zeta^2)^2} - \frac{5}{16}\frac1{(1-\zeta^2)^{\frac52}} + \frac{5}{16}\frac1{(1-\zeta^2)^3} \,.\nn
\end{align}
For $m\geq 2$, $V_m$ belongs to $\QQ\Bigl[\bigl(1-\zeta^2\bigr)^{-\frac12}\Bigr]$ having degree $3m-3$. 
\end{lemma}
\noindent Note that Lemma \ref{lemmae0bessel} implies that as $\nu\rightarrow +\infty$, 
\beq
j_{\nu}\Bigl( \frac{\nu^2 \zeta^2}4 \Bigr)\;\sim\;  \bigl(2-\tfrac1\nu\bigr)^{\nu-\frac12} \,  e^U \,, \quad 
U \= \nu \, U_0 \+ U_1 \+ \frac{U_2}{\nu} + \frac{U_3}{\nu^2} \+ \dots \,, 
\eeq
where $\zeta\in(-1,1)$ is fixed, and 
\begin{align}
& U_0 \= -1 \+ \sqrt{1-\zeta^2} \, - \, \log \bigl(1+\sqrt{1-\zeta^2}\bigr) \,, \quad U_1 \= \frac12 \+ \frac12 \log\bigl(1+\sqrt{1-\zeta^2}\bigr)  -\frac14 \log(1-\zeta^2) \,, \nn\\
& U_{m} \= V_{m}  ~ (m\geq 2)\,. \nn
\end{align}
It is also easy to see that for $\ell\geq 1$, $U_{2\ell+1} + \frac1{\ell(2\ell+1) 2^{2\ell+2}}$ belongs to $\bigl(1-\zeta^2\bigr)^{-\frac{2\ell+1}2} \cdot \QQ\Bigl[\bigl(1-\zeta^2\bigr)^{-\frac12}\Bigr]$.
We omit further details of the proof of Lemma \ref{lemmae0}.

Due to Lemma \ref{lemmae0}, $a_m(\lambda,q)$, $c_m(\lambda,q)$ can be identified with their formal expansions in~$\sqrt{q}$. (Indeed, these 
series are convergent for $2|\sqrt{q}|<|\lambda|$).
Therefore, the large $\lambda$ asymptotic of $M(\frac{\lambda}\e;\frac{\sqrt{q}}\e)$ could be identified with the $\e \rightarrow 0$ (double scaling) asymptotic (identification between elements in $\QQ[[\lambda^{-1},\e,\sqrt{q}]]$). 
Theorem \ref{e0limit} then follows from Lemma \ref{lemmae0}. In particular, the identity 
$\sum_{g\geq 0} \e^{2g-2+k} H_{k}^{[g]}(\lambda_1,\dots,\lambda_k;q) = F_k(\lambda_1,\dots,\lambda_k; \e,q)$ is 
understood as an equality between formal power series in $\e,\sqrt{q}, \lambda_1^{-1},\dots,\lambda_k^{-1}$.

To show the statement for $k=1$, observe that 
$$
H_1^*(z,s) \= \frac{1+G(z;s)}{2} \, \frac{\p \log J_{z-\frac12}(2s)}{\p z} \+ \frac{1-G(z;s)}2 \, \frac{\p \log J_{\frac12+z}(2s)}{\p z}
$$
where we have used \eqref{GBB1}--\eqref{GBB2}. Then the theorem follows from Lemmata~\ref{lemmae0bessel}~and~\ref{lemmae0}.
\epf

We remark that, due to Lemma \ref{lemmae0bessel} the asymptotic formula \eqref{asymD} can also be viewed as an $\e\rightarrow 0$ limit, with 
$a=\frac{\lambda_1}\e$, $b=\frac{\lambda_2}\e$ and $s=\frac{\sqrt{q}}\e$.

\subsection{The $\e\rightarrow \infty$ asymptotic. Proof of Theorem~\ref{einftylimit}} First consider $k\geq 2$. 
For $|\e|>2|\lambda|$ we have
\begin{align}
& \frac12G\Bigl(\frac{\lambda}\e;\frac{\sqrt{q}}\e\Bigr) -\frac12 \= \sum_{k=0}^\infty \A_{2k}(\lambda,q) \e^{-2k}\,,\label{geinf1} \\
& \frac{\sqrt{q}}{\lambda+\frac\e2} \, \widetilde G\Bigl(\frac{\lambda}\e;\frac{\sqrt{q}}\e\Bigr) \= \sqrt{q} \, \sum_{m=0}^\infty \C_m(\lambda,q) \, \e^{-m}\,.\label{geinf2}
\end{align}
where $\A_{2k}$, $\C_m$ are polynomials in $q,\lambda$.
Note that the right hand sides are also the $\e\rightarrow \infty$ asymptotics of the left hand sides. 
The polynomials $\A_{2k},\C_m$ satisfy the following homogeneity conditions:
\begin{align}
& {\rm gr}\,  \A_{2k} \= 2k \, \A_{2k}\,,\quad  {\rm gr}\,  \C_m  \= (m-1) \, \C_m\,.\nn
\end{align}
The first several of these polynomials can be read off from 
\begin{align} 
& \frac12G\Bigl(\frac{\lambda}\e;\frac{\sqrt{q}}\e\Bigr) -\frac12 \= 
 -\frac{4 q}{\e^2} - \frac{16 q \bigl(\lambda^2 - \frac13 q\bigr)}{ \e^4}-\frac{64 q \bigl( \lambda^4 -\frac{10}{27} \lambda^2 q + \frac2{45} q^2 \bigr)}{ \e^6} \+ \dots \,, \nn\\
& \frac1{\lambda+\frac\e2} \, \widetilde G\Bigl(\frac{\lambda}\e;\frac{\sqrt{q}}\e\Bigr) \= \frac{2}{\e} \, - \, \frac{4\lambda}{\e^2}  \+ \frac{8(\lambda ^2- \frac23 q)}{ \e^3} 
\,- \, \frac{ 16( \lambda^3 - \frac29 \lambda q)}{\e^4}  \nn\\
& \qquad\qquad \qquad \qquad \qquad  \+ \frac{ 32(\lambda ^4-\frac{20}{27} \lambda ^2 q+\frac{2}{15} q^2)}{ \e^5} \,-\,
 \frac{64(\lambda^5 - \frac{20}{81}  \lambda ^3 q + \frac{2}{75} \lambda  q^2)}{\e^6} \+  \dots  \nn 
 \end{align}
where $|\e|>2|\lambda|$ is assumed.  Theorem \ref{einftylimit} follows from \eqref{geinf1}--\eqref{geinf2} and the definition of $H_k$. 
The $k=1$ statement easily follows from \eqref{H1sum} and \eqref{H1H1}.
\epf

\subsection{The $q\rightarrow 0$ asymptotic. Proof Theorem~\ref{q0limit}} By definition,
\begin{align}
& \frac12G\Bigl(\frac{\lambda}\e;\frac{\sqrt{q}}\e\Bigr) -\frac12 \= \frac12\sum_{i=0}^\infty \binom{2i}{i}  \frac{q^i}{ \prod_{\ell=0}^{2i-1}\bigl[\lambda+(\ell-i+1/2)\e\bigr]  } \,-\,\frac12 \,,\nn \\
&  \frac{\sqrt{q}}{\lambda+\frac\e2} \, \widetilde G\Bigl(\frac{\lambda}\e;\frac{\sqrt{q}}\e\Bigr) \= \sqrt{q} \, \sum_{i=0}^\infty \binom{2i}{i} \frac{q^i}{\prod_{\ell=0}^{2i}\bigl[\lambda+(\ell-i+1/2)\e\bigr]}\,. \nn
\end{align}
So the definition itself gives the $q\rightarrow 0$ asymptotic of the entries of $M(\frac{\lambda}\e;\frac{\sqrt{q}}\e)$; the coefficients are clearly rational functions of $\lambda$, $\e$. 
Theorem \ref{q0limit} then follows from the definition of $H_k$,~$k\geq 2$. For the case $k=1$, the definition of $H_1$ automatically gives the 
$q\rightarrow 0$ asymptotic, which simplified to \eqref{H1q0asym}.
 \epf
 
\begin{cor} $\forall\,k\geq 2$, the following formulas hold true
\begin{align}
& \sum_{d\geq 0} q^d \, H_{k,d}(\lambda_1,\dots,\lambda_k;\e) \= \sum_{g\geq 0} \e^{2g-2+2k} H_{k}^{[g]}(\lambda_1,\dots,\lambda_k;q)\,, \nn\\
&\sum_{d\geq 0} q^d \, H_{k,d}(\lambda_1,\dots,\lambda_k;\e) \= \sum_{g\geq 0} \e^{-2g} H_{k,[g]}(\lambda_1,\dots,\lambda_k;q) \,. \nn
\end{align}
\end{cor}

Moreover, the following two identities hold true in the corresponding formal series rings:   
\begin{align}
& \sum_{d=1}^\infty q^d \,  \frac{(2d-1)!}{d!^2 \, \prod_{j=1}^d \bigl(\lambda^2-\frac{(2j-1)^2}{4}\e^2\bigr)} \= \sum_{g\geq 0} \e^{2g} H_{1}^{[g]}(\lambda;q) \,,\nn\\
& \sum_{d=1}^\infty q^d \,  \frac{(2d-1)!}{d!^2 \, \prod_{j=1}^d \bigl(\lambda^2-\frac{(2j-1)^2}{4}\e^2\bigr)} \= \sum_{g\geq 0} \e^{-2g} H_{1,[g]}(\lambda;q) \,. \nn
\end{align}
\noindent This corollary indicates that the $q\rightarrow 0$ limit connects the $\e\rightarrow 0$ and the $\e\rightarrow \infty$ limits.

\subsection{The $q\rightarrow \infty$ asymptotic. Proof of Theorem~\ref{qinflimit}}
Recall from \cite{Watson} that 
for any fixed value of $\nu$, as $|y|\rightarrow\infty$ in a sector $|{\rm arg} \, y|\leq\pi-\delta$, the following asymptotic holds true:
\begin{align}
J_\nu(y) \; \sim \; \sqrt{\frac2{\pi y}} \, & \biggl(\cos\bigl(y-\frac{\pi}2\nu-\frac\pi4\bigr)\sum_{m=0}^\infty \frac{(-1)^m\, (\nu-2m+\frac12)_{4m}}{(2m)! \, (2y)^{2m}} \nn\\
& \quad - \sin\bigl(y-\frac{\pi}2\nu-\frac\pi4\bigr)\sum_{m=0}^\infty \frac{(-1)^m\, (\nu-2m-\frac12)_{4m+2}}{(2m+1)! \, (2y)^{2m+1}}\biggr) \,. \label{yinf}
\end{align}

For $k\geq 2$, using \eqref{yinf} and \eqref{GBB1} we have 
\begin{lemma} For $z$ fixed and $|s|\rightarrow \infty$ with $\bigl|{\rm arg}\, s \bigr|\leq\pi-\delta$, we have the asymptotic expansions
\begin{align}
& G(z;s)  \, \sim   \, \frac{\cos(4s)}{\cos(\pi z)} \sum_{r=0}^\infty   \frac{ d_{2r}(z) }{s^{2r}}  
\+   \frac{\sin(4s)}{\cos(\pi z)} \sum_{r=0}^\infty \frac{ d_{2r+1}(z)}{s^{2r+1}} \, -\, \tan(\pi z)
\sum_{r=0}^\infty \frac{\binom{2r}r\prod_{j=-r}^r (z+j)}{2^{4r+1} \, s^{2r+1} } \,, \nn \\
& \frac{s}{z+\frac12} \widetilde G(z;s)  \, \sim  \,    \frac{\sin(4s)}{\cos(\pi z)} \sum_{r=0}^\infty   \frac{ e_{2r}(z) }{s^{2r}}  
\+   \frac{\cos(4s)}{\cos(\pi z)} \sum_{r=0}^\infty \frac{ e_{2r+1}(z)}{s^{2r+1}} \, -\, \tan(\pi z )
\sum_{r=0}^\infty \frac{\binom{2r}r\prod_{j=1-r}^r (z+j)}{2^{4r+1}  \, s^{2r} }  \,, \nn
\end{align}
with explicitly known polynomials $d_r(z)\in \QQ[z^2]$, $e_r(z)\in \QQ[z(z+1)]$.
\end{lemma} 
The rest of Theorem \ref{qinflimit} follows from the definition of $H_k$ as well as elementary trigonometric identities.
In a similar way, one proves statement for $k=1$. 
\epf

\begin{remark}
We would like to mention the following formal solution $W$ to equations \eqref{topoB}--\eqref{trdetB}:
\beq
W \= \begin{pmatrix}  \frac12 - \sqrt{-1} \, w_1 (\lambda,\e,q)  &  \sqrt{-1} \, w_2 (\lambda-\e,\e,q)  
\\   - \sqrt{-1} \, w_2 (\lambda,\e,q) & \frac12   + \sqrt{-1} \, w_1 (\lambda,\e,q)     \end{pmatrix} 
\eeq
where 
$ w_1 (\lambda,\e,q) = \sum_{m=0}^\infty  \frac{(2m-1)!! \, \prod_{j=-m}^m (\lambda+\e j)}{2^{3m+2} \,m!\, q^{m + 1/2}}$, $~ w_2 (\lambda,\e,q) = \sum_{m=0}^\infty  \frac{(2m-1)!! \, \prod_{j=-(m-1)}^{m} (\lambda +  \e j)}{2^{3m+1} \, m!\, q^m}$.
Clearly, $W$ belongs to $\QQ[\lambda,\e][[q^{-1/2}]]$. However, analytic aspect of this formal solution is unclear to us. For example, 
we do not know if there exists an analytic solution satisfying \eqref{asympB} and with the large $q$ asymptotic given by $W$.
We will consider this problem in a subsequent publication.  
\end{remark}

\section{Further remarks} \label{section5}
\subsection{Bispectrality}
Bispectrality is an interesting and rare phenomenon in the theory of integrable systems. 
Let $\bigl(u_n({\bt}),v_n({\bt})\bigr)$ be a solution to the Toda Lattice Hierarchy, and $R_n=R_n(\lambda;\bt)$ its matrix resolvent. 
Denote $\R(\lambda;x,\bt;\e)=R_{x/\e}(\lambda;\bt)$, also called the matrix resolvent. 
We say that the matrix resolvent has {\it bispectrality} if there exists a non-zero scalar function $g(\lambda;\e)$ and an invertible matrix-valued function $A(\lambda;\e)$  such that 
$g(\lambda;\e)\,A(\lambda;\e) \, \R(\lambda;x,{\bf 0};\e)\,A(\lambda;\e)^{-1}$ is a function of $h(\lambda,x)$ and $\e$ only for some scalar function $h$.
This type of bispectrality can be defined analogously to 
other integrable system (where in most cases $\e$ can be taken to be 1 for simplicity).  
For the Toda Lattice Hierarchy, one might guess that the GUE \cite{DY1} and the $\mathbb{P}^1$ cases are essentially (modulo some group actions)
two only possible cases  
possessing bispectrality of the above type, but 
there are not enough evidences for supporting this guess. So classifying this type of bispectrality for the Toda Lattice Hierarchy seems still to be 
 an open question.
Also, bispectrality looks still mysterious. Indeed,
we do not know its origin. 
We call the solution has the {\it type-I bispectrality} if the function $h(\lambda,x)=\lambda-x$.
Conjecturally, the so-called ``topological" solution to an integrable system always has the type-I bispectrality. 
We hope to study criterion of bispectrality beyond type-I in a future publication (the method given in \cite{DG} might be helpful). 

\noindent {\bf Conjecture}.
{\it Let $M$ be a semisimple (calibrated) Frobenius manifold. 
Assume that the integrable hierarchy of topological type of $M$ \cite{Du1, DZ-norm, Gi} admits a Lax pair formalism. 
Then, a solution of this integrable hierarchy is topological \emph{iff} its matrix resolvent possesses bispectrality of Type I. The 
same statement is valid for the Hodge hierarchy \cite{DLYZ} of $M$.}

Note that validity of the Main Conjecture for GW invariants of $\mathbb{P}^1$ is confirmed in this paper.
\begin{prop}
The Conjecture is true for ADE singularities.
\end{prop}
\pf
The necessity part is precisely the Lemma 4.2.3 of \cite{BDY3} where it is called the Key Lemma. The sufficiency part 
follows from the uniqueness theorem of topological ODEs \cite{DY1}, i.e. the space of solutions regular at infinity is equal to the rank of the simple 
Lie algebra. 
\epf

\subsection{Dual topological ODE}
The topological difference equation for $\mathbb{P}^1$ can be written as
\beq
M(z-1;s)  \, A  \,- \, A \,   M(z;s) \= z \, M(z-1;s)  \, B  \,-\,  z \, B \,   M(z;s) 
\eeq
with 
\beq 
A \= \begin{pmatrix} \frac1 2 & s\\ -s & 0\end{pmatrix} \,, \qquad B \= \begin{pmatrix} 1 & 0\\ 0 & 0\end{pmatrix} \,. 
\eeq
\begin{defi}
The dual topological ODE for the Toda Lattice Hierarchy associated with the solution corresponding to the GW invariants of $\mathbb{P}^1$ (within the stationary sector) is defined by
\beq\label{dual}
e^{y} \, \widetilde M  \, A  \,- \, A \,   \widetilde M \=  
e^{y} \biggl( \widetilde M + \frac{d \widetilde M}{d y}  \biggr) \, B   \, -  \, B \,   \frac{d \widetilde M}{d y} 
\eeq
where $\widetilde M=\widetilde M(y;s)$ is a matrix-valued function in $y$, and $s$ is an arbitrary parameter. 
\end{defi}
Topological and dual topological equations \eqref{m1xp},  \eqref{dual} are related 
via a Laplace type transform, i.e.
$$
\widetilde M (y;s) \= \frac1{2\pi i} \int_\gamma e^{z \, y} \, M(z;s) \, d z
$$
where $\gamma$ is an appropriate contour on the complex $z$ plane.

\subsection{Analytic invariants of~$\mathbb{P}^1$} We have already seen that the formal series $\e \, F_k(\lambda_1,\dots,\lambda_k; 0, \e,q)$, defined as the generating series  
of the GW invariants $\langle \tau_{i_1}(\phi_{\alpha_1}) \cdots \tau_{i_k}(\phi_{\alpha_k})  \rangle_{g,d}$ of~$\mathbb{P}^1$ 
(in {\it full genera} and of {\it all degrees})
is not convergent as a series of $\e$, or as a (multi-)series of $\lambda_1^{-1},\dots,\lambda_k^{-1}$. However, as a power 
series of $q$, it does converge, which gives the motivation of defining the analytic $k$-point functions 
$H_k\bigl(\frac{\lambda_1}\e, \dots, \frac{\lambda_k}\e; \frac{q^{1/2}}\e\bigr)$, such that the GW invariants are 
the coefficients in the full asymptotic of the double scaling limit $\e\rightarrow 0$ (or of the $q\rightarrow0$ limit) of 
$H_k\bigl(\frac{\lambda_1}\e, \dots, \frac{\lambda_k}\e; \frac{q^{1/2}}\e\bigr)$. The definition of~$H_k$ is certainly natural, and provides the 
non-perturbative version of topological quantum field theory for~$\mathbb{P}^1$.  
We refer to the coefficients in the $\e\rightarrow \infty$ asymptotics (or again in the $q\rightarrow 0$ asymptotic but with 
$|\e|\gg|\lambda_i|$, $i=1,\dots,k$), and in the $q\rightarrow\infty$ asymptotics of 
$H_k\bigl(\frac{\lambda_1}\e, \dots, \frac{\lambda_k}\e; \frac{q^{1/2}}\e\bigr)$ as {\it analytic invariants} of~$\mathbb{P}^1$. 
These invariants are counterparts of the GW invariants. 
For example, the first few $H_{k,[g]}$ are
\begin{align}
& H_{1,[0]} \= 0\,, \quad H_{1,[1]} \= -4q\,,  \quad H_{1,[2]} \= -16 q \lambda^2 + \frac83 q^2\,, 
\quad H_{1,[3]} \= -64 q \lambda^4 + \frac{320}{27} q^2 \lambda^2 - \frac{128}{135} q^3\,, \nn\\
& H_{2,[1]} \= 16q \,, \quad H_{2,[2]} \= 64 q \bigl(\lambda_1^2+\lambda_2^2\bigr) \,-\, \frac{256}9 q^2 \,, \nn\\
& H_{2,[3]} \=  256 q \bigl(\lambda_1^4+ \lambda_1^2 \lambda_2^2 + \lambda_2^4 \bigr) 
 \,-\, \frac{256}{81} q^2 \bigl(37\lambda_1^2-2\lambda_1\lambda_2+37 \lambda_2^2 \bigr)  \+ \frac{53248}{2025} q^3 \,,  \nn \\
& H_{3,[1]} \= -64 q \,, \quad H_{3,[2]} \= -256 q \bigl(\lambda_1^2+\lambda_2^2+\lambda_3^2\bigr) \+ \frac{6656}{27} q^2 \,, \nn\\
& H_{3,[3]} \= -1024 q\, \bigl( \lambda_1^4+\lambda_2^4+ \lambda_3^4 +\lambda_1^2\lambda_2^2+\lambda_2^2\lambda_3^2+\lambda_3^2\lambda_1^2\bigr) \,, \nn\\
&\qquad\qquad \+ \frac{4096}{243} q^2 \Bigl(59(\lambda_1^2+\lambda_2^2+\lambda_3^2) -(\lambda_1\lambda_2+\lambda_2\lambda_3+\lambda_3\lambda_1)\Bigr) \,- \, \frac{13027328}{30375} q^3\,,\nn
\end{align}
where we recall that $H_{k,[g]}$ are defined in the expansion \eqref{asympHeinf}.
These are the counterparts for $\e\rightarrow \infty$.

We also list the first several $H_k^{d,m}$ and $\widetilde H_k^{d,m}$ for $k\geq 2$:
\begin{align}
& H_2^{0,0} \= 
\= \e^2 \, \frac{\frac12 - C_1 \, C_2 - S_1 \, S_2}{(\lambda_1-\lambda_2)^2} \, ,  \quad H_2^{1,0} \= 0\,, \quad  H_2^{1,1} \= - \e^2 \, \frac{S_1 - S_2}{4 \, (\lambda_1-\lambda_2)} \,,  
\quad \widetilde H_2^{1,1} \= 0 \,, \nn\\
& H_2^{2,0}  \=  \frac{1}{32} \Bigl[\e^2 - 2 \, \e^2 \, S_1 \, S_2 -2 \, (\lambda_1+\lambda_2)^2\Bigr] \,, \quad H_2^{2,1} \= 0\,,  \quad \widetilde H_2^{2,1} =  \e \, \frac{\bigl(\e^2-2 \lambda_1^2\bigr) \, S_2 - \bigl(\e^2-2 \lambda_2^2\bigr) \, S_1}
{16 \, (\lambda_1-\lambda_2)}\,, \nn\\
& H_2^{2,2} \= \frac{\e^2}{32} \,,  \quad H_2^{3,0} \=0 \,, \quad H_2^{3,1} \= \frac{\bigl(\e^4- \e^2 \bigl(\lambda_1^2+2 \lambda_2^2\bigr)+\lambda_2^4\bigr) S_1 - 
\bigl(\e^4-\e^2 \bigl(2 \lambda_1^2+\lambda_2^2\bigr)+\lambda_1^4\bigr) S_2}{32 \, (\lambda_1-\lambda_2)}\,, \nn\\
&  \widetilde H_2^{3,1} \=0\,,\quad  H_2^{3,2} \=0 \,, \quad \widetilde H_2^{3,2} \= \e\,\frac{\e^2-(\lambda_1^2+\lambda_2^2)}{64}   \,,   \nn \\
& H_3^{0,0} \=0 \,, \quad H_3^{1,0} \= - \e^2 \, \frac{ (\lambda_2^2-\lambda_3^2) \, S_1 + (\lambda_3^2-\lambda_1^2) \, S_2 
+(\lambda_1^2-\lambda_2^2) \, S_3}{4 \, (\lambda_1-\lambda_2)  (\lambda_2-\lambda_3) (\lambda_3-\lambda_1)} \,, \nn \\
& H_3^{1,1} \= 0\,, \quad  \widetilde H_3^{1,1} \= \e^3 \, \frac{ (\lambda_1-\lambda_2)  \, S_1 \, S_2
+ (\lambda_2-\lambda_3) \, S_2\, S_3 +(\lambda_3-\lambda_1) \, S_3\, S_1}
{4 \, (\lambda_1-\lambda_2) (\lambda_2-\lambda_3) (\lambda_3-\lambda_1)} \,. \nn
\end{align}
Here, $C_1=\cos \bigl(\frac{\pi  \lambda_1}{\e}\bigr)$, $C_2=\cos \bigl(\frac{\pi  \lambda_2}{\e}\bigr)$, $S_1=\sin \bigl(\frac{\pi  \lambda_1}{\e}\bigr)$, $S_2=\sin \bigl(\frac{\pi  \lambda_2}{\e}\bigr)$, $S_3=\sin \bigl(\frac{\pi  \lambda_3}{\e}\bigr)$.
For $k=1$, we have
\begin{align}
& H_1^{*,1}=0 \,, \quad H_1^{*,2} = \frac{2\lambda^2-\e^2}{16} \,, \quad H_1^{*,3}=0 \,, 
\quad H_1^{*,4} = -\frac{2\lambda^6-16 \e^2 \lambda^4+32 \e^4\lambda^2-9\e^6}{384\e^2} \,, \nn\\
& \widetilde H_1^{*,1} = \frac\e4 \,, \quad \widetilde H_1^{*,2} \= 0\, , \quad  
\widetilde H_1^{*,3} = - \frac{\lambda^4-3\e^2\lambda^2+\e^4}{32 \e} \,, \quad \widetilde H_1^{*,4} = 0 \,. \nn
\end{align}
Here, $H_k^{d,m}$ and $H_1^{*\,d}$ are defined in \eqref{asympHq} and \eqref{asympHq}, respectively, which give the counterparts for $q\rightarrow \infty$.
 It will be interesting to study the Stokes phenomenon
of the GW invariants by investigating the asymptotic of $H_k\bigl(\frac{\lambda_1}\e, \dots, \frac{\lambda_k}\e; \frac{q^{1/2}}\e\bigr)$ 
as $\e$ goes to 0 within different sectors.

\medskip
\medskip
\medskip
\medskip

\noindent Boris Dubrovin

\noindent SISSA, via Bonomea 265, Trieste 34136, Italy

\noindent dubrovin@sissa.it

\medskip
\medskip

\noindent Di Yang

\noindent Max-Planck-Institut f\"ur Mathematik, Vivatsgasse 7, Bonn 53111, Germany

\noindent diyang@mpim-bonn.mpg.de

\medskip
\medskip

\noindent Don Zagier

\noindent Max-Planck-Institut f\"ur Mathematik, Vivatsgasse 7, Bonn 53111, Germany

\noindent dbz@mpim-bonn.mpg.de

\end{document}